\documentclass[12pt]{amsart}
\usepackage{amsmath,amscd,amsthm,amsfonts, amssymb,amsxtra}
\usepackage[all]{xy}
\usepackage{graphics}
\usepackage[margin=.5in,font=small,labelfont=sc]{caption}
\SelectTips{cm}{}
\newtheorem{thm}{Theorem}[section]
\newtheorem*{un-no-thm}{Theorem}
\newtheorem{cor}[thm]{Corollary}     
\newtheorem{lem}[thm]{Lemma}         
\newtheorem{prop}[thm]{Proposition}

\newtheorem{bigthm}{Theorem}

\newtheorem{bigconj}[bigthm]{Conjecture}

\theoremstyle{definition}
\newtheorem{defn}[thm]{Definition}   

\theoremstyle{definition}

\theoremstyle{definition}

\theoremstyle{remark}
\newtheorem{rem}[thm]{Remark}

\newtheorem{ex}[thm]{Example}

\DeclareMathOperator*{\holim}{holim\, }
\DeclareMathOperator*{\hocolim}{hocolim\, }
\DeclareMathOperator*{\colim}{colim\, }

\DeclareMathOperator{\rel}{rel\, }
\DeclareMathOperator{\Trunc}{Trunc}
\DeclareMathOperator{\hofiber}{hofiber}
\DeclareMathOperator{\Arr}{Arr}
\DeclareMathOperator{\Map}{Map}
\DeclareMathOperator{\Aut}{Aut}
\DeclareMathOperator{\ext}{Ext}

\begin{document}
\title[Multiple disjunction]
{Multiple disjunction for spaces \\ of  Poincar\'e embeddings}
\date{\today}
\author{Thomas G. Goodwillie}
\address{Brown University, Providence, RI 02912}
\email{tomg@math.brown.edu}
\author{John R. Klein}
\address{Wayne State University, Detroit, MI 48202}
\email{klein@math.wayne.edu}
\begin{abstract} We obtain multirelative connectivity
statements about spaces of Poincar\'e embeddings, as
precursors to analogous statements about spaces of smooth
embeddings. The latter are the key to convergence results
in the functor calculus approach to spaces of embeddings.
\end{abstract}
\thanks{Both authors are partially supported by the NSF}
\maketitle
\setlength{\parskip}{1pt plus 0pt minus 1pt} \def\:{\colon}
\def\Bbb{\mathbb} \def\bold{\mathbf} \def\cal{\mathcal}

\setcounter{tocdepth}{1}
\tableofcontents

\section{\label{intro} Introduction}

This paper addresses some questions about Poincar\'e complexes, mainly
in order to answer analogous questions about smooth manifolds. We
begin by informally explaining the questions.

Let $E(P,N)$ be the space of smooth embeddings of $P$ in $N$. If $Q$
is a submanifold of $N$ then by a general-position argument one sees
readily that the inclusion map
$$
E(P,N-Q)\rightarrow E(P,N)
$$
is $(n-p-q-1)$-connected, where $n$, $p$, and $q$ are the dimensions.

We are after multirelative generalizations of this relative
connectivity result. For example, if $Q_1$ and $Q_2$ are disjoint
submanifolds of $N$ then there is the triad
$$
(E(P,N);E(P,N-Q_1),E(P,N-Q_2))
$$ of spaces of embeddings, and we can ask about the vanishing of
low-dimensional triad homotopy groups. In other words, we can look at
the square diagram
$$
\SelectTips{cm}{}
\xymatrix{
E(P,N-(Q_1\cup Q_2)) \ar[r]\ar[d] &
E(P,N-Q_2) \ar[d]\\
E(P,N-Q_1) \ar[r] & E(P, N)\
}
$$ and ask if there is a range, depending only on the dimensions
$(n,p,q_i)$, in which it is a homotopy pullback square. More generally
for disjoint submanifolds $(Q_1,\dots,Q_r)$ there is the $(r+1)$-ad
$$
(E(P,N);E(P,N-Q_1),\dots,E(P,N-Q_r))
$$
with its associated $r$-dimensional cubical diagram, formed by the spaces
$$
\cap_{i\in S}E(P,N-Q_i)= E(P,N-Q_S),
$$ where $Q_S$ is the union of the $Q_i$ for $i\in S\subset \underline
r =\lbrace 1,\dots r\rbrace$. Call a cubical diagram of spaces
$k$-cartesian if the canonical map from the first space to the
homotopy limit of the rest of the diagram is $k$-connected. The
expected statement is that the cube just described is
$$(1-p+\Sigma_i(n-q_i-2))\text{-cartesian.}$$ (We prefer the language
of cubical diagrams to that of $(r+1)$-ads, except for concessions to
tradition in the case of manifold or Poincar\'e duality
$(r+1)$-ads. In the case of an $r$-cube arising from an $(r+1)$-ad,
\lq $k$-cartesian\rq\ basically means that the homotopy groups of the
$(r+1)$-ad vanish through dimension $k+r-1$, plus some $\pi_0$
information.  See Appendix B or \cite[\S1]{Goodwillie_CALC2}.)

While the case $r=1$ can be settled by simple dimension-counting, sharp results in the general case require more elaborate methods. Our strategy, which succeeds whenever the codimensions $n-p$ and $n-q_i$ are all at least three, is briefly this:

\begin{itemize}
\item {\bf Step One \rm (homotopy).} Solve an analogous problem with manifolds
replaced by Poincar\'e complexes.
\item {\bf Step Two \rm (surgery).} Pass from Poincar\'e embeddings to
block embeddings using a version of the Browder-Casson-Sullivan-Wall
theorem.
\item {\bf Step Three \rm (pseudoisotopy).} Pass from block embeddings to genuine embeddings using the multirelative generalization of Morlet's disjunction lemma in \cite{thesis}.
\end{itemize}
Step one is carried out in the present paper. Steps two and three will
appear in \cite{Good_Klein}.
\medskip

Before turning to the details, here are a few remarks.
\medskip

First, as to what this kind of multirelative connectivity statement is
good for: it is used in applications of functor calculus to embedding
problems, for example, to establish convergence of Weiss' Taylor tower
\cite{Weiss}.  See also \cite{Good_Weiss1}, \cite{Good_Weiss2},
\cite{Good_Klein_Weiss}.

Second, in many cases the number $1-p+\Sigma_i(n-q_i-2)$ is negative,
so that the statement has no content. In applications one compensates
for small codimension by choosing $r$ to be large.

Third, weaker statements can be proved with a good deal less work. For
example, the square displayed above can be shown to be
$(2n-2p-q_1-q_2-3)$-cartesian (as opposed to $2n-p-q_1-q_2-3$) by an
argument involving dimension-counting plus the Blakers-Massey triad
connectivity theorem. This is explained in detail and generalized to
any number of $Q_i$ in Appendix B, \ref{weak}.

Let us now go into a little more detail about definitions and
hypotheses. We choose to work with embeddings fixed along the boundary
of the ambient manifold. Let $P$ be a manifold triad. Thus it has a
boundary which is the union of two parts, $\partial_0 P$ and
$\partial_1 P$, meeting at a corner set
$\partial_{01}P=\partial\partial_0 P=\partial\partial_1 P$. (Any of
these sets might be empty or disconnected.) An embedding $e_0$ of
$\partial_0 P$ in $\partial N$ is fixed in advance. A point in
$E(P,N)$ is an embedding $e$ of $P$ in $N$ that extends $e_0$ and that
is transverse to $\partial N$ at each point of
$\partial_0P=e^{-1}(\partial N)$ (cf.\ fig.\ 1 below). Note that, although $e_0$ will
rarely be mentioned explicitly, $E(P,N)$ depends on $e_0$.
\begin{figure}
\begin{minipage}{2.9in}
\includegraphics{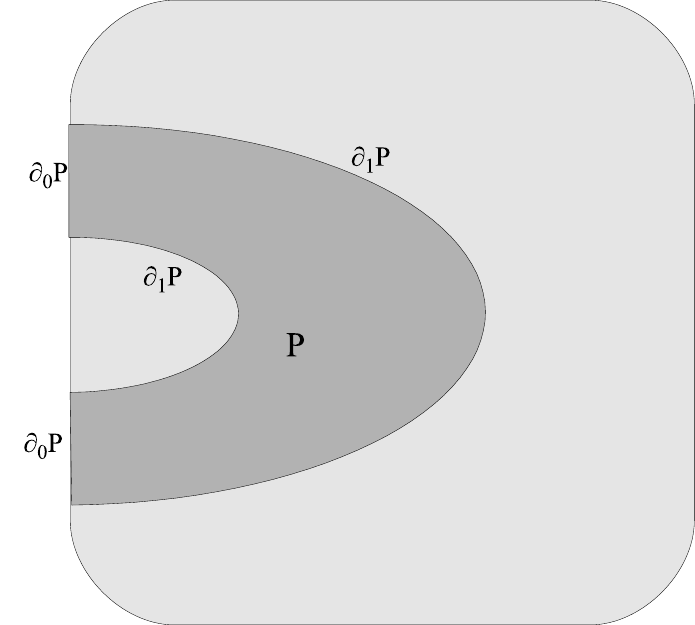}
\end{minipage}
\caption{\small An
embedding of $P$ in $N$. The lighter region
indicates the complement.}
\end{figure}

To avoid side issues related to normal bundles and their Poincar\'e
analogues, we prefer to work in codimension-zero settings as much as
possible. The numbers $p$ and $q_i$ in the statements below may be
thought of informally as the dimensions of spines of $P$ and
$Q_i$. The precise assumption (in the manifold case) is that $P$ has
{\it handle dimension} $\leq p$ relative to $\partial_0P$ in the sense
that $P$ can be built up from a collar $I\times\partial_0P$ by
attaching handles of index at most $p$. The most important case to
think of, apart from the case when $P$ is $p$-dimensional, is the case
when $P$ is an $(n-p)$-disk bundle over a $p$-dimensional manifold and
$\partial_0P$ is the restriction of the bundle to the boundary, or to
part of the boundary. The same considerations apply to $Q_i$,
$\partial_0Q_i$, and $q_i$.

The main statement that we are after is Conjecture \ref{conj1}
below. Let $N$ be a smooth compact $n$-dimensional manifold. Let
$(Q_1,\dots,Q_r)$, $r\geq 1$, be pairwise disjoint compact
submanifolds of $N$ transverse to $\partial N$ and let $q_i$ be the
handle dimension of $Q_i$ relative to $\partial_0 Q_i=Q_i\cap \partial
N$. Let $(P,\partial_0 P,\partial_1P)$ be a compact manifold triad and
let $p$ be the handle dimension of $P$ relative to $\partial_0
P$. Suppose that an embedding $e_0:\partial_0 P\rightarrow \partial N$
is given, disjoint from $\partial_0Q_i$ for all $i$. For $S\subset
\underline r = \lbrace 1,\dots r\rbrace$ let $Q_S=\cup_{i\in S}Q_i$.

\begin{bigconj}\label{conj1} In this situation the $r$-dimensional cubical diagram formed by the spaces $\lbrace E(P,N-Q_S)\rbrace$ is $(1-p+\Sigma_i(n-q_i-2))$-cartesian.
\end{bigconj}

In \cite{Good_Klein} this will be proved in all cases where $n-p\geq
3$ and $n-q_i\geq 3$.

Of course, this statement covers the case when $P$ is $p$-dimensional
and $Q_i$ is $q_i$-dimensional, because handle dimension is always
less than or equal to dimension. On the other hand, a simple argument
involving the replacement of submanifolds by tubular neighborhoods
shows that in proving the conjecture it is enough to consider the
opposite extreme, when $P$ and $Q_i$ are $n$-dimensional. This is the
case we will work with.

We now explain a different form of the same conjecture. Let $N$ be a
smooth compact $n$-dimensional manifold. Assume $s\geq 2$. For $1\leq
i\leq s$ let $(Q_i,\partial_0 Q_i,\partial_1Q_i)$ be a smooth compact
manifold triad. Suppose that the manifolds $\partial_0Q_i$ are
disjointly embedded in $\partial N$. For $S\subset \underline s$ let
$Q_S$ be the disjoint union $\coprod_{i\in S}Q_i$ and let $E(Q_S,N)$
be the space of embeddings fixed on $\partial_0Q_S=\coprod_{i\in
S}\partial_0 Q_i$

\begin{bigconj} \label{conj2} In this situation
the $s$-dimensional cubical diagram formed by
the spaces $\lbrace E(Q_S,N)\rbrace$ is $(3-n+\Sigma_i(n-q_i-2))$-cartesian.
\end{bigconj}

Let us see that Conjecture \ref{conj2} for $s=r+1$ is equivalent to
Conjecture \ref{conj2} for $r$. We can assume that each $Q_i$ is
$n$-dimensional. First consider the simplest case, $r=1$ and
$s=2$. Given $N$, $Q_1$, and $Q_2$ with $\partial_0Q_1$ and
$\partial_0Q_2$ disjointly embedded in $\partial N$, consider the
square of fibrations
$$
\xymatrix{
E(Q_{\lbrace 1,2\rbrace},N)\ar[r]\ar[d] & E(Q_2,N)\ar[d]\\E(Q_1,N)\ar[r] &E(\emptyset,N)\
}
$$
For each choice $e$ of a point in $E(Q_1,N)$, the horizontal maps
give a map between the fibers of the vertical maps:
$$
E(Q_2,N-e(Q_1))\rightarrow E(Q_2,N)
$$ The square is $k$-cartesian if and only if for every choice of $e$
this map of fibers is $k$-connected. The case $r>1$ goes the same way,
because more generally (see \ref{fiber-cube-connectivity} below) a map
of $r$-cubes is a $k$-cartesian $(r+1)$-cube if and only if for every
choice of basepoint the resulting $r$-cube of homotopy fibers is
$k$-cartesian. Note that if one of the $Q_i$ is renamed $P$ then
$(3-n+\Sigma_i(n-q_i-2))$ becomes $(1-p+\Sigma_i(n-q_i-2))$.

The main result of this paper is stated below, somewhat informally, as
Theorem \ref{disjunction}. (The precise statement is Theorem
\ref{ISS}.) It is essentially a Poincar\'e analogue of Conjecture \ref{conj1}
above. It is not an exact analogue: the connectivity is lower by
one. We expect that the conclusion should hold without that slight
weakening. We have settled for the weaker result because it saves a
lot of trouble in the proof and because we know how to obtain the
sharp result in the smooth case from this not-quite-sharp result in
the Poincar\'e case. We limit ourselves to the codimension zero case
because that suffices for the purpose at hand, and we assume $n-p$ and
$n-q_i$ are at least three because the proof requires it.

In the Poincar\'e case, rather than working with handles we make the
following definition: An $n$-dimensional Poincar\'e triad
$(P;\partial_0 P,\partial_1P)$ has \it homotopy spine dimension\rm\
$\leq p$ relative to $\partial_0 P$ if
\begin{enumerate}
\item the pair $(P,\partial_0 P)$ is homotopically $p$-dimensional,
and
\item the pair $(P,\partial_1 P)$ is $(n-p-1)$-connected.
\end{enumerate}
Condition (1) means that the pair is a retract, up to weak
equivalence, of a CW pair with relative dimension $\leq p$. See
Appendix A for a fuller discussion. For a manifold triad the homotopy
spine dimension is less than or equal to the handle dimension.

A weak form of (1) would say that $(P,\partial_0 P)$ is
cohomologically $p$-dimensional (has no cohomology above dimension $p$
for any coefficient system). A weak form of (2) would say that
$(P,\partial_1 P)$ is homologically $(n-p-1)$-connected (has no
homology below dimension $n-p$ for any coefficient system). Of course,
by duality the weak (1) is equivalent to the weak (2). The weak (2)
implies (2) if the pair $(P,\partial_1 P)$ is $2$-connected, and if
$p\geq2$ the weak (1) implies (1), by \ref{coh-htpy}. Thus there is
some redundancy here.

\begin{bigthm} \label{disjunction} Let $(N,\partial N)$ be a Poincar\'e
pair of formal dimension $n$. Let $(P;\partial_0P,\partial_1P)$ and
$(Q_i;\partial_0Q_i,\partial_1Q_i)$ be Poincar\'e triads of formal dimension $n$,
where $i$ runs from $1$ to $r\geq 1$. Let $p$ and $q_i$ be the homotopy spine
dimensions of $P$ relative to $\partial_0P$ and of $Q_i$ relative
to $\partial_0Q_i$, and assume that these numbers are all less than or equal to $n-3$.
Suppose that the $Q_i$ are disjointly embedded in $N$ with
$\partial_0 Q_i=Q_i\cap \partial N$ and that $\partial_0P$ is
embedded in $\partial N$ disjointly from all the $\partial_0 Q_i$.
Then the cubical diagram
formed by the spaces of Poincar\'e embeddings
$\lbrace E^h(P,N-Q_S)\rbrace$ is $(-p+\Sigma_i(n-q_i-2))$-cartesian.
\end{bigthm}

There is an equivalent formulation of this result, related to it as Conjecture \ref{conj2} is related to Conjecture \ref{conj1}. We will not write it down here.
It remains to explain what we mean by spaces of (codimension-zero)
Poincar\'e embeddings. The details are in section 2, but here is a
preliminary account:

If $\partial N$ is a closed Poincar\'e complex of dimension $n-1$,
consider the category $(\partial N\downarrow\cal T)$ of spaces under
$\partial N$. Let $w(\partial N\downarrow\cal T)$ be the subcategory
having all the objects but having only weak equivalences as
morphisms. Let $\cal I^h_n(\partial N)$ be the full subcategory of
$w(\partial N\downarrow\cal T)$ whose objects $\partial N\rightarrow
X$ satisfy $n$-dimensional relative Poincar\'e duality. The nerve of
this, or of an equivalent small category, can then be called a \lq
space of interiors\rq\ for $\partial N$. If $(N,\partial N)$ is a CW
Poincar\'e pair then the component of the nerve of $\cal
I^h_n(\partial N)$ determined by $N$ is a delooping of the topological
monoid $\Aut^h(N)$ of all homotopy equivalences
$N\rightarrow N$ fixing $\partial N$ pointwise.

Now suppose that $(N;\partial_0 N,\partial_2 N)$ and $(P;\partial_0
P,\partial_1 P)$ are $n$-dimensional Poincar\'e triads and that
$\partial_0 P=\partial_0 N$. Then $\partial_1 P\cup\partial_2 N$ is an
$(n-1)$-dimensional Poincar\'e space and there is a functor
$$
\cal I^h_n(\partial_1 P\cup\partial_2 N)\rightarrow \cal
I^h_n(\partial_0 P\cup\partial_2 N)=\cal I^h_n(\partial N)
$$ giving by \lq gluing to $P$ along $\partial_1P$\rq. This can be
described as a two-step process: there are functors
$$
w(\partial_1 P\cup\partial_2 N\downarrow\cal T)\rightarrow
w(P\cup\partial_2 N\downarrow\cal T)\rightarrow
w(\partial N \downarrow\cal T),
$$
resulting from the \lq inclusion\rq\ maps
$$
\partial_1 P\cup\partial_2 N\rightarrow P\cup\partial_2 N\leftarrow \partial_0 P\cup\partial_2 N=\partial N.
$$
The first step is pushforward along the first inclusion, and the
second is composition with the other inclusion. The composed functor
takes the subcategory $\cal I^h_n(\partial_1 P\cup\partial_2 N)$ into
the subcategory $\cal I^h_n(\partial N).$

The homotopy fiber, with respect to the point $N\in\cal I^h_n(\partial
N)$, of this map of spaces of interiors is what we take as the
definition of the space $E^h(P,N)$ of Poincar\'e embeddings. The
following imprecise remark is offered as an indication that this is a
good definition: $\Aut^h(N)$ may be thought of as acting on $E^h(P,N)$,
and for any given Poincar\'e embedding of $P$ in $N$ with complement
$C$ there is a fibration sequence $\Aut^h(C)\rightarrow
\Aut^h(N)\rightarrow E^h(P,N)$. We will not be using this statement,
but the reader is welcome to make it more precise and to extract a
proof of it from some of the results in section 2 below.

The best justification for the definition is that on the one hand we
can prove a theorem about it here while on the other hand we can use
that theorem in \cite{Good_Klein} to deduce a theorem about spaces of
smooth embeddings.

In order to work with these \lq spaces of interiors\rq\ and the
relevant maps between them, we need to think systematically about
categories of spaces under a fixed space.

Specifically, we need to establish something called the {\it
analyticity of pushforward,} which may be loosely stated as follows:
Fix an integer $d\geq 2$. To a space $A$ is associated, as before, the
category $w(A\downarrow\cal T)$ of spaces under $A$ and weak
equivalences. Define a full subcategory by requiring of an object $X$
that the structural map $A\rightarrow X$ must be both homotopically
$d$-dimensional and a cofibration. Consider this as a functor of $A$ by observing that a map $A\rightarrow B$ induces a suitable functor $X\mapsto
B\cup_AX$ from spaces under $A$ to spaces under $B$. Passing to realizations of nerves, we obtain a space-valued
functor of the space $A$ which preserves weak equivalences and (this
is the point) is analytic in the sense of \cite{Goodwillie_CALC2} and
\cite{Goodwillie_CALC3}.  For example, when this functor is applied to
a homotopy pushout square
$$
\xymatrix{
A\ar[r]\ar[d] & A_1\ar[d]\\
A_2 \ar[r] & A_{12}
}
$$
in which the maps $A\rightarrow A_i$ are $k_i$-connected, it yields
a $(k_1+k_2-d)$-cartesian square, provided $k_i\geq 2$. The statement
generalizes to higher-dimensional cubical diagrams, the general
formula being $2-d+\Sigma(k_i-1)$. One issue not addressed in this rough
sketch is that we appeared to need a nerve of a non-small
category. Another is that we did not precisely make a functor, since
the pushforward along a composition is merely isomorphic, not equal,
to a composition of pushforward functors.

Incidentally, in discussing such matters, we will sometimes say that a
category or a functor is $k$-connected, or that a cubical diagram of
categories is $k$-cartesian. Of course we mean that by making
realizations of nerves one gets a $k$-connected space or map of
spaces, or a $k$-cartesian cube of spaces.

In order to work with categories of spaces under a fixed space, and
especially to describe the homotopy fibers of certain maps between
(realizations of nerves of) these, we are led to consider something
more general: categories of factorizations of a fixed map of spaces.

Two kinds of comments about rigor are worth making here.

First, because of the subject matter (multirelative or cubical
connectivity), we must be particularly careful to avoid the kinds of
sloppiness that can hide behind phrases like \lq up to equivalence\rq\
or \lq up to homotopy\rq. Diagrams must commute on the nose. For
example, it is relatively useless to know that a square diagram
$$
\xymatrix{
A\ar[r]\ar[d] & B\ar[d]\\
C\ar[r] & D
}
$$
commutes up to homotopy unless one also chooses a particular
homotopy between the two maps from $A$ to $D$. Questions like \lq is
it a homotopy pullback square?\rq\ have no meaning without such a
choice. Keeping track of the necessary choices of that kind for more
complicated diagrams does not hold much appeal for us, hence the
insistence on strictly commutative diagrams.

An effect of this careful attention to strict commutativity is that
not only the proof but even the detailed statement of Theorem
\ref{ISS} is more technical than we could wish. The result about
manifolds to be derived from it in \cite{Good_Klein} will look better.

Second, the paper has been written in a free and easy way as far as
set theory is concerned. Something must be done to guard against such
monstrosities as a space that is the realization of the nerve of the
category of all spaces. After all, we are defining a space of
Poincar\'e embeddings to be the homotopy fiber of a map between two
things of just that kind. Here are some suggestions about how to work around these difficulties.

Option One: Fix a Grothendieck universe $\cal U$, and use only $\cal
U$-sets to make the spaces which are the objects of the category $\cal
T$. Various results in the paper refer to spaces that arise from
nerves of categories closely related to $\cal T$. These are not spaces
in the same sense, but our statements and proofs are legitimate when
they are understood to refer to a larger universe. It is reasonable to
ask whether the various \lq large\rq\ spaces encountered along the way
are in fact weakly homotopy equivalent to \lq small\rq\ spaces. The
answer is yes in many cases, including components of factorization
categories and those spaces of Poincar\'e embeddings that are relevant
to our argument. One disadvantage of this approach is that it assumes
the existence of $\cal U$.

Option Two: Place some more drastic restriction on the cardinalities
of the spaces in $\cal T$, but not so drastic as to exclude any of the
spaces you really care about. To do this carefully while still using
model category language would mean giving up on model category
language, or at least weakening the model category axioms about
existence of (co)limits.

Option Three (our own favorite): Learn to call a category $k$-connected without making a
space out of it: $\cal C$ is $k$-connected if for every integer $m$
from $0$ through $k+1$, for every small subcategory $\cal
C_1\subset\cal C$, for every map of $S^{m-1}$ into the realization of
the nerve of $\cal C_1$, there exists a small category $\cal C_2$ of
$\cal C$ containing $\cal C_1$ such that in the realization of its
nerve the map extends to $D^m$. One might as well be a little more general: call a simplicial class $X$ $k$-connected if for every integer $m$
from $0$ through $k+1$, for every simplicial set $X_1\subset X$, for every map of $S^{m-1}$ into the realization of $ X_1$, there exists a larger simplicial set $X_2\subset X\dots $. Similar considerations apply to maps
being $k$-connected, and to cubical diagrams being $k$-cartesian. The
reader can easily verify that the standard tools like the \lq higher
Blakers-Massey theorem\rq\ still apply.

The remainder of the paper is organized as follows:

Section 2 is about categories of factorizations and their nerves with
respect to weak equivalence. It gives a number of general facts and a
number of key examples, including those that lead to the definition of
spaces of Poincar\'e embeddings and some others that are needed in
proofs.

Section 3 presents a technique, \it homology truncation\rm, that can
be used to obtain (multi-)relative connectivity information about such
spaces of factorizations.

Section 4 introduces the principle of analyticity of pushforward and
indicates how to prove it, in the case of squares, using the
truncation method.

Section 5 gives the proof of our main result in the simplest case,
$r=1$ and $s=2$, using the analyticity for squares.

Section 6 states and proves the analyticity in rigorous fashion. The
proof is the extension to higher cubes of a method presented very
sketchily for squares in section 4.

Section 7 proves the main result. The proof is an extension to higher
cubes of the method described in section 5 for squares.

Thus in some sense sections 6 and 7 contain no new ideas, but just
technical details. The technicalities in section 6 are rather
extensive, but those in section 7 amount to little more than
introducing notation for some subdivided cubical diagrams.

The brief Appendix A gives conventions and basic facts about
$k$-connectedness of maps and homotopical $d$-dimensionality of maps.

Appendix B recalls conventions and basic results about cubical
diagrams, including the \lq higher Blakers-Massey theorems\rq.

Appendix C develops what we need from the theory of Poincar\'e duality
spaces in a self-contained, self-indulgent, and, we hope, easy-to-read
way.

\section{\label{factorizations} Spaces of Factorizations}

Let $\cal T$ be the category of topological spaces. If $j\: A\rightarrow
B$ is a morphism in $\cal T$, then $\cal T(j)$ is the category of
factorizations of $j$. An object consists of a space $X$ and maps
$i\:A\rightarrow X$ and $p\:X\rightarrow B$ such that $p\circ i=j$. A
morphism $g\:(X,i,p)\rightarrow (X',i',p')$ is a map $g\: X\rightarrow
X'$ such that $g\circ i=i'$ and $p'\circ g=p$. We sometimes write $X$
for $(X,i,p)$, and we sometimes write $\cal T(A\rightarrow B)$ when
$j$ is understood. Note that $\cal T(A\rightarrow *)$ is isomorphic to
$(A\downarrow\cal T)$, the category of spaces under $A$, while $\cal
T(\emptyset\rightarrow B)$ is isomorphic to $(\cal T\downarrow B)$,
the category of spaces over $B$.

The theory of model categories will not be used heavily in this paper,
but it will provide a useful language. Let $\mathcal T$ have the
Quillen model structure \cite{Quillen}, based on weak homotopy
equivalences and Serre fibrations. Each category $\mathcal T(j)$
inherits a model structure from $\mathcal T$ in the simplest way
imaginable: a morphism is a weak equivalence (cofibration, fibration)
in $\cal T(j)$ if it is a weak equivalence (cofibration, fibration) of
spaces. The fact that this leads to a model structure is completely
general.

Let $w\mathcal T(j)$ be the subcategory of $\mathcal T(j)$ having all
of the objects but having only weak equivalences as morphisms.

Various useful homotopy types will arise as nerves of portions of these categories $w\cal T(j)$. Whenever we use the word \lq portion\rq\, it will refer to a subcategory that is the union of some components of the larger category. (One might say that it is \lq open and closed\rq.)

\begin{ex}This is the space of all possible interiors for a given boundary. If $Y$ is a Poincar\'e space of formal dimension $n-1$, let $\cal I^h_n(Y)$ be the portion of $w\cal T(Y\rightarrow *)$ consisting of those maps $Y\rightarrow X$ that satisfy relative $n$-dimensional duality.
\end{ex}

\begin{ex} \label{left-inverses}
Given maps $A\rightarrow B\rightarrow C$, let $w\cal T(A\rightarrow B; \sim C)$ be the portion of $w\cal T(A\rightarrow B)$ consisting of those factorizations $A\rightarrow X\rightarrow B$ such that the composed map $X\rightarrow B\rightarrow C$ is a weak equivalence.  It turns out  (\ref{function-space}) that if $A\rightarrow C$ is a cofibration and $B\rightarrow C$ is a fibration then this is equivalent to the space of right inverses of $B\rightarrow C$ in the category of spaces under $A$. Likewise, let $w\cal T(B\rightarrow C; A\sim )$ be the portion of $w\cal T(B\rightarrow C)$ consisting of factorizations $B\rightarrow X\rightarrow C$ such that $A\rightarrow X$ is a weak equivalence. If $A\rightarrow C$ is a fibration and $A\rightarrow B$ is a cofibration then this is equivalent to the space of left inverses of $A\rightarrow B$ in the category of spaces over $C$.
\end{ex}

\begin{ex}Let $w\cal T(A\rightarrow B,d)$ be the portion of $w\cal T(A\rightarrow B)$ consisting of objects $X$ such that the map $A\rightarrow X$ is homotopically $d$-dimensional (a retract up to equivalence of a relative cell complex whose cells have dimension at most $d$).
\end{ex}

\begin{ex} \label{ur-completion}
This is the space of all ways of completing a given diagram
$$
\xymatrix{
A\ar[d]\\
B\ar[r] & D
}
$$
to a homotopy pushout square. That is, it is the portion of $w\cal T(A\rightarrow D)$ consisting of factorizations $A\rightarrow X\rightarrow D$ such that the (commutative) square
$$
\xymatrix{
A\ar[d]\ar[r] & X\ar[d]\\
B\ar[r] & D
}
$$
is a homotopy pushout. We will see (Corollary \ref{trunc_cor})
 that if the map $B\rightarrow D$ is homotopically $d$-dimensional with $d\geq 2$ and the map $A\rightarrow B$ is $k$-connected with $k\geq 2$ then the space in question is $(k-d)$-connected (in particular nonempty if $k\geq d-1$). The tool for proving this, \it homology truncation\rm, will be explained in the next section.
\end{ex}

\begin{ex} \label{cocart-replace} In much the same vein, given a square
$$
\xymatrix{
A\ar[d]\ar[r] & B\ar[d]\\
C\ar[r] & D
}
$$
we may consider the portion of $w\cal T(A\rightarrow B\times_D C)$ consisting of factorizations $A\rightarrow X\rightarrow B\times_D C$ such that the square
$$
\xymatrix{
X\ar[d]\ar[r] & B\ar[d]\\
C\ar[r] & D
}
$$
is a homotopy pushout. Again, a truncation argument yields good information about it, at least if the maps $B\rightarrow D\leftarrow C$ are fibrations. An extension of this argument to high-dimensional cubes is at the crux of the proof of analyticity of pushforward, which in turn is the key to our main result.
\end{ex}

Returning to spaces of factorizations in general: Let $w\cal T^c(j)\subset w\cal T(j)$ be the full subcategory of cofibrant objects, in other words those $(X,i,p)$ for which $i$ is a cofibration. Similarly let $w\cal T^f(j)$ be the category of fibrant objects and $w\cal T^{cf}(j)$ the intersection of these.

\begin{prop} \label{replacements} The inclusions
$$
\xymatrix{
w\cal T^{cf}(j)\ar[r]\ar[d] & w\cal T^f(j)\ar[d]\\
w\cal T^c(j)\ar[r] & w\cal T(j)
}
$$
are weak equivalences. That is, they induce weak equivalences of nerves.
\end{prop}

\proof
Functorial factorization of a morphism as a trivial fibration composed with a cofibration yields a functor $w\cal T(j)\rightarrow w\cal T^c(j)$ which is, up to natural transformation, a two-sided inverse to the inclusion functor. The same method covers the other inclusions.
\endproof

The next result describes the homotopy type of any component of the nerve of $w\cal T(j)$ as a delooping of the simplicial monoid of homotopy automorphisms of a suitable object in that component. To state and prove it, we introduce a simplicial category $w_{\bullet}\cal T(j)$ with $w_0\cal T(j)=w\cal T(j)$. For each $m\geq 0$, the objects of $w_m\cal T(j)$ are the same as those of $\cal T(j)$. A morphism from $X$ to $Y$ in $w_m\cal T(j)$ is a family of $w\cal T(j)$-morphisms parametrized by the
simplex $\Delta^m$, continuously in the sense that the corresponding map $\Delta^m\times X\rightarrow Y$ of spaces is continuous. For an object $X$ of $\cal T(j)$, let $C_Xw_{\bullet}\cal T(j)$ be the component of $w_{\bullet}\cal T(j)$ containing the object $X$, in other words the full simplicial subcategory of objects weakly equivalent to $X$. Let $\Aut_{\bullet}^h(X)$ be the simplicial monoid of endomorphisms of $X$, in other words the full simplicial subcategory of $w_{\bullet}\cal T(j)$ having sole object $X$.

\begin{lem} \label{Waldhausen_adaptation}
The inclusion $w\cal T(j)=w_0\cal T(j)\rightarrow w_{\bullet}\cal T(j)$ is a weak equivalence (i.e., it induces a weak equivalence of nerves). If $X$ is fibrant and cofibrant, then the inclusion $\Aut_{\bullet}^h(X)\rightarrow C_Xw_{\bullet}\cal T(j)$ is also a weak equivalence.
\end{lem}

\proof This is adapted from \cite{W}. For the first statement it will suffice to show that for every $m$ the degeneracy map $s\: w\cal T(j)=w_0\cal T(j)\rightarrow w_m\cal T(j)$ is a weak equivalence. We may restrict attention to the full subcategories $w_0\cal T^c(j)$ and $w_m\cal T^c(j)$ of cofibrant objects, by (a slight extension of) \ref{replacements}. Let $d\: w_m\cal T^c(j)\rightarrow w\cal T^c(j)$ be the map corresponding to the vertex $v_0$ of $\Delta^m$. Then $d\circ s$ is the identity and we have to show that $s\circ d$ is weakly homotopic to the identity. For this it is enough to exhibit a functor $F\: w_m\cal T^c(j)\rightarrow w_m\cal T^c(j)$ admitting natural maps from the identity and from $s\circ d$. Let $F$ send the object $A\rightarrow X\rightarrow B$ to
$$
A\rightarrow A\cup_{A\times I}(X\times I)\rightarrow B.
$$
Define $F$ on morphisms by means of a homotopy between the identity map $\Delta^m\rightarrow \Delta^m$ and the constant map to $v_0$. (The assumption that $A\rightarrow X$ is a cofibration is needed to ensure that the weak homotopy type of $A\cup_{A\times I}(X\times I) $ is not pathological.)

For the second statement, first replace $C_Xw_m\cal T(j)$ by the full subcategory $C_Xw_m\cal T^{cf}(j)$; again this leaves the homotopy type unchanged by
\ref{replacements}. In $C_Xw\cal T^{cf}(j)$ any morphism $f\: Y\rightarrow Z$ is a strong homotopy equivalence in the sense that there exists $g\: Z\rightarrow Y$ such that $f\circ g$ (resp.  $g\circ f$) and $1_Z$ (resp. $1_Y$) are the faces of a morphism in $w_1\cal T(j)$. The conclusion now follows by
\cite[prop.\ 2.2.5]{W}
or \cite[\S2]{DK}.
\endproof

Note that the simplicial monoid $\Aut_{\bullet}^h(X)$ is grouplike (that is, the monoid $\pi_0\Aut_{\bullet}^h(X)$ is a group), again because in the case of cofibrant and fibrant $X$ every level-zero morphism has a homotopy inverse. Thus its nerve is a delooping.

The category $\cal T(j)$ depends (almost) functorially on the map $j$ in two different ways, both of which will play a role. A square
$$
\xymatrix{
A_0\ar[r]\ar[d] & B_0\ar[d]\\
A_1\ar[r] & B_1
}
$$
determines a \it pushforward\rm\ functor
$$
\xymatrix {
\cal T(A_0\rightarrow B_0)\ar[r]  & \cal T(A_1\rightarrow B_1)}
$$
sending $X$ to $A_1\cup_{A_0}X$ and a \it pullback\rm\ functor
$$
\xymatrix{
\cal T(A_0\rightarrow B_0)& \cal T(A_1\rightarrow B_1)\ar[l]
}
$$
sending $Y$ to $B_0\times_{B_1}Y$.
These functors form a Quillen adjoint pair.

The pushforward does not preserve weak equivalences in general, but of course it does if it is applied only to cofibrant objects or if $A_0\rightarrow A_1$ is a cofibration. Thus if $A_0\rightarrow A_1$ is a cofibration we have a functor
$$
\xymatrix{
w\cal T(A_0\rightarrow B_0)\ar[r] & w\cal T(A_1\rightarrow B_1)
}
$$
and if not we still have a functor
$$
\xymatrix{
w\cal T^c(A_0\rightarrow B_0)\ar[r] & w\cal T^c(A_1\rightarrow B_1),
}
$$
which in view of \ref{replacements} is just as good for many purposes.

Likewise the pullback does not in general preserve weak equivalences, but it does if it is applied only to fibrant objects or if $B_0\rightarrow B_1$ is a fibration.

The reason for the word \lq almost\rq\ a few paragraphs back is that, for example, a composition of pushforward functors
$$
\xymatrix{
T(A_0\rightarrow B_0)\ar[r] & \cal T(A_1\rightarrow B_1)\ar[r] & \cal T(A_2\rightarrow B_2)
}
$$
is not equal, but merely isomorphic, to a pushforward functor
$$
\xymatrix{
T(A_0\rightarrow B_0)\ar[r] & \cal T(A_2\rightarrow B_2)\, .
}
$$
This failure of strict functoriality is something that we have to cope with.

We are now ready to define the space of Poincar\'e embeddings of $P$ in $N$. Suppose that $(P;\partial_0P,\partial_1P)$ and $(N;\partial_0N,\partial_2N)$ are $n$-dimensional Poincar\'e triads, with $\partial_0 N=\partial_0 P$. The maps
$$
\xymatrix{
\partial_1P\cup_{\partial_{01}P}\partial_2N\ar[r] &
P\cup_{\partial_{01}P}\partial_2N &
\partial_0 P\cup_{\partial_{01}P}\partial_2N =\partial N\ar[l]
}$$
induce pushforward and pullback functors:
$$
w\cal T(\partial_1P\cup_{\partial_{01}P}\partial_2N\rightarrow *)\rightarrow
w\cal T(P\cup_{\partial_{01}P}\partial_2N\rightarrow *)\rightarrow
w\cal T(\partial N\rightarrow *).
$$
The composition carries $\cal I_n^h(\partial_1P\cup_{\partial_{01}P}\partial_2N)$ into $\cal I_n^h(\partial N)$.

\begin{defn}The space $E^h(P,N)$ is the homotopy fiber, with respect to the point $N$, of (the map of realizations of nerves induced by) the functor
$$
\cal I_n^h(\partial_1P\cup_{\partial_{01}P}\partial_2N)\rightarrow\cal I_n^h(\partial N)
$$
defined above.
\end{defn}

Returning to general statements about pushforward and pullback, we have the following:

\begin{rem} \label{addend_remark} If in the map of maps $(A_0\rightarrow B_0)\rightarrow (A_1\rightarrow B_1)$ the maps $A_0\rightarrow A_1$ and $B_0\rightarrow B_1$ are weak equivalences, then the resulting Quillen adjoint  pair is a Quillen equivalence. In particular the nerves of $w\cal T(A_0\rightarrow B_0)$ and $w\cal T(A_1\rightarrow B_1)$ are then weakly equivalent.
\end{rem}

In many cases Quillen's ``Theorem B''
can be used to identify the homotopy fiber of a pushforward functor as the nerve of a category. The left Quillen fiber (comma category) of the pushforward functor
$$
\xymatrix{
\cal T(A_0\rightarrow B_0)\ar[r] & \cal T(A_1\rightarrow B_1)
}
$$
with respect to an object $X_1\in \cal T(A_1\rightarrow B_1)$ is (categorically) equivalent to $\cal T(A_0\rightarrow X_1\times_{B_1}B_0)$.  If $A_0\rightarrow A_1$ is a cofibration, so that the restricted pushforward functor $$
\xymatrix{
w\cal T(A_0\rightarrow B_0)\ar[r] & w\cal T(A_1\rightarrow B_1)
}
$$
is defined, then its left Quillen fiber is the portion of ${w\cal T(A_0\rightarrow X_1\times_{B_1}B_0)}$ consisting of objects $X_0$ such that the resulting map $A_1\cup_{A_0}X_0\rightarrow X_1$ is a weak equivalence. To apply
``Theorem B,'' we need to know that for any map $X_1\rightarrow X_1'$ in $w\cal T(A_1\rightarrow B_1)$ the induced functor between left Quillen fibers is a weak equivalence. This is the case, by a slight extension of
\ref{addend_remark}, as long as the induced map $X_1\times_{B_1}B_0\rightarrow X'_1\times_{B_1}B_0$ is an equivalence. In particular it is the case if $B_0\rightarrow B_1$ is a fibration. We conclude:

\begin{prop} \label{hofiber-pushforward}
If the maps $A_0\rightarrow A_1$ and $B_0\rightarrow B_1$ are respectively a cofibration and a fibration, then the homotopy fiber of the pushforward functor $w\cal T(A_0\rightarrow B_0)\rightarrow w\cal T(A_1\rightarrow B_1)$ with respect to an object $X_1$ of $\cal T(A_1\rightarrow B_1)$ is weakly equivalent to the portion of $w\cal T(A_0\rightarrow X_1\times_{B_1}B_0)$ consisting of objects $X_0$ such that the resulting map $A_1\cup_{A_0}X_0\rightarrow X_1$ is a weak equivalence.
\end{prop}

There is a dual statement, with a dual proof:

\begin{prop} \label{hofiber-pullback} With the same hypothesis, the homotopy fiber of the pullback functor $w\cal T(A_1\rightarrow B_1)\rightarrow w\cal T(A_0\rightarrow B_0)$ with respect to an object $X_0$ of $\cal T(A_0\rightarrow B_0)$ is weakly equivalent to the portion of $w\cal T(A_1\cup_{A_0}X_0\rightarrow B_1)$ consisting of objects $X_1$ such that the resulting map $X_0\rightarrow X_1\times_{B_1}B_0$ is a weak equivalence.
\end{prop}

\begin{rem}The fact that we need a left Quillen fiber to describe the homotopy fiber of pushforward and a right Quillen fiber to describe the homotopy fiber of pullback is a hindrance to giving a simple description of a space of Poincar\'e embeddings as the nerve of a category.
\end{rem}

We record separately the form that the last Proposition takes in the special case when the map $B_0\rightarrow B_1$ is an identity map $1_B$. In this case the pullback functor $w\cal T(A_1\rightarrow B)\rightarrow w\cal T(A_0\rightarrow B)$ is simply a forgetful functor, composition with $A_0\rightarrow A_1$. Changing the names of the spaces from $(A_0,A_1,B,X)$ to $(A,C,D,B)$, we have a statement about a square diagram of spaces
$$
\xymatrix{
A\ar[r]\ar[d] & B\ar[d]\\
 C\ar[r] & D
 }
$$

\begin{prop} \label{pre-function-space} If $A\rightarrow C$ is a cofibration, then the homotopy fiber of the pullback functor $w\cal T(C\rightarrow D)\rightarrow w\cal T(A\rightarrow D)$ with respect to the object $B$ of $\cal T(A\rightarrow D)$ is weakly equivalent to $w\cal T(C\cup_{A}B\rightarrow D,B\sim )$.
\end{prop}

We next show that in this same special case there is also a useful description of the homotopy fiber as a function space, under the additional assumption that $B\rightarrow D$ is a fibration. This will justify an assertion made in Example \ref{left-inverses}.

Notice that in the most naive sense the fiber over $B$ of this pullback functor is precisely the set of all solutions  of the lifting/extension problem posed by the square,
in other words the set of continuous maps $C\rightarrow B$ that make the two triangles commute. Let $L$ be the simplicial set in which a $p$-simplex is a family of such maps parametrized continuously by $\Delta^p$ in the sense that the associated map $\Delta^p\times C\rightarrow B$ is continuous.

\begin{prop} \label{function-space}
In the square above, assume that $A\rightarrow C$ is a cofibration and that $B\rightarrow D$ is a fibration. Then the homotopy fiber with respect to $B$ of the pullback map $w\cal T(C\rightarrow D)\rightarrow w\cal T(A\rightarrow D)$ is equivalent to the space $L$ of lifting/extensions defined above.
\end{prop}

\proof To any factorization $A\rightarrow X\rightarrow D$ such that $X\rightarrow D$ is a fibration, we can functorially associate a simplicial set $\hat L(X)$, defined just like $L$ but with $X$ substituted for $B$. By obstruction theory the functor $\hat L$ preserves weak equivalences. Consider the following bisimplicial set $\tilde L$, a one-sided bar construction. $\tilde L_{p,q}$ is the coproduct, over all $q$-simplices $X_0\rightarrow\dots\rightarrow X_q$ in the nerve of $w\cal T^f(B\rightarrow D,B\sim)$, of $\hat L_p(X_0)$.

On the one hand the diagonal simplicial set of $\tilde L$ is the homotopy colimit of a diagram, a diagram of simplicial sets with indexing category $w\cal T^f(B\rightarrow D,B\sim)$. This is a diagram in which every arrow is a weak equivalence, and in which the indexing category has initial object $B$, so $\tilde L$ is equivalent to $\hat L(B)=L$.

On the other hand, $\tilde L$ can also be seen as the nerve of a simplicial category $F_{\bullet}$. An object of $F_p$ consists of a space $X$, a factorization $A\rightarrow X\rightarrow D$ such that $X\rightarrow D$ is a fibration, and a $\Delta^p$-parametrized family of weak equivalences $B\rightarrow X$ compatible with the maps to $D$ and the maps from $A$. For each $p$ the degeneracy functor $F_0\rightarrow F_p$ is  a weak equivalence, by an argument like that in the proof of \ref{Waldhausen_adaptation}. Therefore the inclusion $F_0\rightarrow F_{\bullet}$ is a weak equivalence. This completes the proof, because $F_0$ is the portion of $w\cal T(C\cup_{A}B\rightarrow D)$ that has already been identified with the homotopy fiber of pullback, or rather it is  the equivalent corresponding portion of $w\cal T^f(C\cup_{A}B\rightarrow D)$.
\endproof

To sum up, there is a chain of weak equivalences
$$
\xymatrix{
\hofiber
(w\cal T(C\rightarrow D)\rightarrow w\cal T(A\rightarrow D)) \\
w\cal T(C\cup_{A}B\rightarrow D,B\sim )\ar[u] \\
w\cal T^f(C\cup_{A}B\rightarrow D,B\sim )\ar[u]\ar[d] \\
\text{\rm lifting/extensions}
}
$$
Of course, each of these equivalences is suitably natural.

\begin{ex} \label{example-lift-extend}
The space $E^h(P,N)$ of Poincar\'e embeddings is by definition the homotopy fiber of the composition (pushforward followed by pullback)
$$
\xymatrix{
\cal I^h_n(\partial_1P\cup\partial_2N)\ar[r] & \tilde{\cal I}^h_n(P\cup\partial_2N)\ar[r] & \cal I^h_n(\partial_0P\cup\partial_2N)=\cal I^h_n(\partial N)
}
$$
Here $\tilde{\cal I}^h_n(P\cup\partial_2N)$ stands for that portion of $w\cal T(P\cup\partial_2N\rightarrow *)$ which is the preimage of the portion $\cal I^h_n(\partial_0P\cup\partial_2N)$ of $w\cal T(\partial_0P\cup\partial_2N\rightarrow *)$. The symbol $\tilde{}$ is a reminder that $P\cup\partial_2N$ is not a Poincar\'e space.
If $\partial_0P\rightarrow P$ is a cofibration, then the homotopy fiber of the second (pullback) map above can be described as the space of solutions of the lifting/extension problem
$$
\xymatrix{
\partial_0P\ar[r]\ar[d] & N\ar[d]\\
P\ar[r] & {*}
}
$$
or in other words the space $F(P,N)$ of all maps $P\rightarrow N$ extending the given map $\partial_0P\rightarrow N$. (To emphasize the boundary condition, we may sometimes denote such a space by $F(P,N \rel \partial_0P)$.) Therefore we have, up to natural equivalence, a map
$E^h(P,N)\rightarrow F(P,N)$. This may be thought of as the \lq inclusion\rq\ of the space of embeddings fixed on $\partial_0P$ into the space of all functions fixed on $\partial_0P$. Its homotopy fiber is equivalent to that of the pushforward map above.
\end{ex}

One simple consequence of \ref{function-space} is the following:

\begin{prop} \label{pullback-functor-diagram} Let
$$
\xymatrix{
A\ar[r]\ar[d] & B\ar[d]\\
C\ar[r] & D
}
$$
be a pushout square of cofibrations. Then the diagram of pullback functors
$$
\xymatrix{
\cal T(A\rightarrow *)& \cal T(B\rightarrow *)\ar[l]\\
\cal T(C\rightarrow *)\ar[u] & \cal T(D\rightarrow *)\ar[l]\ar[u]
}
$$
is $\infty$-cartesian.
\end{prop}

\proof For a (fibrant) object $X$ of $\cal T(C\rightarrow *)$, the homotopy fibers of the two horizontal maps above are weakly equivalent to two isomorphic function spaces $F(B,X
\rel  A)$ and $F(D,X \rel C)$, and of course the two (chains of) equivalences are compatible.
\endproof

The same approach yields:

\begin{prop} \label{pullback-pushforward-diagram} Again let
$$
\xymatrix{
A\ar[r]\ar[d] & B\ar[d]\\
C\ar[r] & D
}
$$
be a pushout square of cofibrations. Assume that $A\rightarrow B$ is homotopically $d$-dimensional and that $A\rightarrow C$ is $k$-connected. Then the diagram
$$
\xymatrix{
\cal T(A\rightarrow *)\ar[d]& \cal T(B\rightarrow *)\ar[l]\ar[d]\\
\cal T(C\rightarrow *) & \cal T(D\rightarrow *)\ar[l]
}
$$
(pullback functors horizontally and pushout functors vertically)
is ${(k-d)}$-cartesian.
\end{prop}

(We are permitting ourselves some laxity here, in that the square commutes only up to canonical isomorphism.)

\proof For a fibrant object $X$ of $\cal T(A\rightarrow *)$, the
homotopy fiber of the upper horizontal map is equivalent to $F(B,X
\rel  A)$. The homotopy fiber of the lower horizontal map is equivalent
to $F(D,C\cup_A X \rel C)=F(B,C\cup_A X \rel A)$. The relevant map
${F(B,X \rel  A)\rightarrow F(B,C\cup_A X \rel A)}$ is $(k-d)$-connected,
since it is induced by a $k$-connected map $X\rightarrow C\cup_A X$.
\endproof

Connectivity questions about the square
$$
\xymatrix{
\cal T(A\rightarrow *)\ar[d]\ar[r] & \cal T(B\rightarrow *)\ar[d]\\
\cal T(C\rightarrow *)\ar[r] & \cal T(D\rightarrow *)
}
$$
of pushout functors require further methods and will be considered later (Theorems \ref{square-pushforwards} and \ref{pushforward-analyticity}).

For completeness we record the duals of \ref{pre-function-space} and \ref{function-space}:

\begin{prop}If $B\rightarrow D$ is a fibration, then the homotopy fiber of the pushforward (forgetful) functor ${w\cal T(A\rightarrow B)\rightarrow w\cal T(A\rightarrow D)}$ with respect to the object $C$ of $\cal T(A\rightarrow D)$ is weakly equivalent to ${w\cal T(A\rightarrow B\times_DC,\sim C)}$.
\end{prop}

\begin{prop}With the hypothesis of \ref{function-space}, the homotopy fiber with respect to $C$ of the pushforward (forgetful) functor $w\cal T(A\rightarrow B)\rightarrow w\cal T(A\rightarrow D)$ with respect to the object $C$ of $\cal T(A\rightarrow D)$ is weakly equivalent to the space $L$ of lifting/extensions $C\rightarrow B$.
\end{prop}

\section{\label{truncation} Homology Truncation}

To motivate the definition of (good) homology truncation below, consider Example \ref{ur-completion}.

By definition, for a square
$$
\xymatrix{
A\ar[r]\ar[d] & X\ar[d]\\
B\ar[r] & D
}
$$
to be a homotopy pushout means that the induced map
\begin{equation} \label{**}
\hocolim(B\leftarrow A\rightarrow X)\rightarrow D
\end{equation}
is a weak equivalence. This is the same as saying that the map is both $2$-connected and a homology equivalence. We are calling a map $Y\rightarrow D$ a homology equivalence if it induces homology isomorphisms for all coefficient systems on $D$. Assume for simplicity that $D$ is based and path-connected with $\pi_1(D)=\Gamma$, so that coefficient systems correspond to left $\mathbb Z\Gamma$-modules. To test for homology equivalence it suffices to consider homology with coefficients in the free module $\mathbb Z\Gamma$.

In this discussion the symbol $C_*(D)$ will mean the singular chain complex of $D$ with coefficients in the local system corresponding to the module $\Bbb Z\Gamma$. This is a complex of (free, right) $\Bbb Z\Gamma$-modules, since the coefficient module in this case is in fact a bimodule. (This same complex of free modules can alternatively be described as the integral chains on the universal covering space of $D$, if $D$ is the sort of space that has a universal covering space.) More generally for any space $Y$ over $D$ the symbol $C_*(Y)$ will mean the singular chain complex of $Y$ with coefficients in that local system pulled back to $Y$. If $Y\rightarrow Z$ is a map of spaces over $D$ then $C_*(Y\rightarrow Z)$ will mean the algebraic mapping cone of the chain map $C_*(Y)\rightarrow C_*(Z)$ and $H_*(Y\rightarrow Z)$ will mean its homology. To repeat, all of this is with coefficients in $\Bbb Z\Gamma$.

Returning to the map \eqref{**}, the homology condition means precisely that the induced homomorphism
$$
H_i(A\rightarrow X)\rightarrow H_i(B\rightarrow D)
$$
is an isomorphism for every $i$. If we assume that the map $A\rightarrow B$ is $2$-connected, then the canonical map
$$
X\rightarrow \hocolim(B\leftarrow A\rightarrow X)
$$
is also  $2$-connected, and therefore the $2$-connectedness of
the map \eqref{**} is equivalent to $2$-connectedness of the map $X\rightarrow D$. Thus in this case the factorizations $A\rightarrow X\rightarrow D$ which yield homotopy pushout squares are precisely those which are good homology truncations (of $D$ relative to $A$ with respect to $C_*(A\rightarrow D)\rightarrow C_*(B\rightarrow D)$) in the sense of the following definition.

\begin{defn}Let $A\rightarrow D$ be a map of spaces with $D$ path-connected and $\pi_1(D)=\Gamma$. Let $K_*$ be a nonnegatively graded chain complex of projective right $\Bbb Z\Gamma$-modules and let $\alpha\: C_*(A\rightarrow D)\rightarrow K_*$ be a $\Bbb Z\Gamma$-linear chain map. A factorization $A\rightarrow X\rightarrow D$ is a \it homology truncation\rm\  of $D$ relative to $A$ with respect to $\alpha$ if the composition
$$
C_*(A\rightarrow X)\rightarrow C_*(A\rightarrow D)\rightarrow K_*
$$
is a chain equivalence. If in addition the map $X\rightarrow D$ is $2$-connected, then it is a {\it good homology truncation.} Define
$$
\Trunc(A\rightarrow D, \alpha)
$$
to be the portion of $w\cal T(A\rightarrow D)$
whose objects are the good homology truncations with respect to $\alpha$.
\end{defn}

\begin{rem}If the given chain map $\alpha$ is $n$-connected, then in any homology truncation the map $X\rightarrow D$ is necessarily \it homologically\rm\ $(n-1)$-connected, since $H_i(X\rightarrow D)\cong H_i(C_*(A\rightarrow X)\rightarrow C_*(A\rightarrow D))\cong H_{i+1}(\alpha)$. In cases of interest to us, $\alpha$ will be $n$-connected for some $n\geq 3$. In such cases the map $X\rightarrow D$ in a good homology truncation will automatically be $(n-1)$-connected, being both $2$-connected and homologically $(n-1)$-connected.
\end{rem}

\begin{rem}Call the chain complex $K_*$ of projective modules cohomologically $d$-dimensional if for every module $M$ the complex \newline $Hom_{\mathbb Z\Gamma}(K_*,M)$ has trivial cohomology in dimensions greater than $d$. (This is the case for example, if $K_*$ is actually $d$-dimensional.) In this case the map $A\rightarrow X$ in a good homology truncation must be cohomologically $d$-dimensional, and therefore (by \ref{coh-htpy}) homotopically $d$-dimensional if $d\geq 2$.
\end{rem}

\begin{thm}Suppose that $n\geq 3$, the chain map $\alpha$ is $n$-connected, $d\geq 2$, and $K_*$ is cohomologically $d$-dimensional. Then $
{\Trunc(A\rightarrow D,\alpha)}$ is $(n-d-1)$-connected.
\end{thm}

\proof This says first of all that  a good homology truncation exists if $n\geq d$. That was proved in \cite{Klein}.
We recall the argument. The map $A\rightarrow D$, like any map of spaces, can be factored as $A\rightarrow Y\rightarrow D$ with $A\rightarrow Y$ a relative cell complex of dimension less than or equal to $n-1$ and $Y\rightarrow D$ an $(n-1)$-connected map. The composed chain map
$$
C_*(A\rightarrow Y)\rightarrow C_*(A\rightarrow D)\rightarrow K_*
$$
is $(n-1)$-connected, because it is the composition of $(n-1)$-connected maps, so its algebraic mapping cone $\bar K_*$ is an $(n-1)$-connected complex of projective $\Bbb Z\Gamma$-modules. Furthermore, $\bar K_*$ is $n$-dimensional in the sense that the cohomology $H^i(\bar K_*;\Lambda)$ vanishes when $i>n$ for any $\Bbb Z\Gamma$-module $\Lambda$, because $K_*$ and $C_*(A\rightarrow Y)$ are respectively $d$-dimensional and $(n-1)$-dimensional in this sense and $d\leq n$. It follows (by considering the case when $\Lambda$ is an injective module, so that $H^i(\bar K_*;\Lambda)=\hom_{\Bbb Z\Gamma}(H_i(\bar K_*),\Lambda)$) that $H_i(\bar K_*)=0$ when $i>n$. The unique nontrivial homology module $H_n(\bar K_*)$ must then be projective since we now have $\ext_{\Bbb Z\Gamma}^{i-n}(H_n(\bar K_*),\Lambda)=H^i(\bar K_*;\Lambda)=0$ for all $i>n$ and all $\Lambda$. This projective module can be arranged to be free by adding a sufficiently large free module to it, and that can be achiev!
 ed geometrically by attaching sufficiently many $(n-1)$-cells to $Y$ (trivially, with trivial map to $D$). At this point the pair $(Y,A)$ is still $(n-1)$-dimensional and the map $Y\rightarrow D$ is still $(n-1)$-connected. The free module can now be eliminated by attaching $n$-cells, as follows: We have surjective maps
$$
\pi_n(Y\rightarrow D)\rightarrow H_n(Y\rightarrow D)\rightarrow H_n(\bar K_*)
$$
Choose a basis for $H_n(\bar K_*)$. Lift each basis element to $\pi_n(Y\rightarrow D)$, and attach cells to $Y$ (with suitable map to $D$) accordingly.

Next, it says that if $n\geq d+1$ then any two good truncations $(X,i,p)$ and $(X,i',p')$ are in the same component of $w\cal T(A\rightarrow D)$. We may assume that the objects are cofibrant and fibrant. Since $i$ is a homotopically $d$-dimensional cofibration and $p'$ is a $d$-connected (because $(n-1)$-connected) fibration, there is a map $(X,i,p)\rightarrow (X,i',p')$ by obstruction theory. Any such map must in fact be a weak equivalence (thus a morphism of $w\cal T(A\rightarrow D)$), because it is a homology equivalence and a $\pi_1$-isomorphism.

Finally, if $n> d+1$ then it says that the loopspace of any component of $\Trunc(A\rightarrow D,\alpha)$ is $(n-d-1)$-connected. In view of
\ref{Waldhausen_adaptation},
this means that for a fibrant and cofibrant object $(X,i,p)$ the space of self-maps is $(n-d-2)$-connected. This is true, again by obstruction theory.
\endproof

The truncation theorem applied to Example \ref{ur-completion} (with $n=k+1$) yields:

\begin{cor} \label{trunc_cor}
If the maps $A\rightarrow B$ and $B\rightarrow D$ are respectively $k$-connected and homotopically $d$-dimensional and if $k\geq 2$ and $d\geq 2$, then the space of all homotopy pushout squares
$$
\xymatrix{
A\ar[r]\ar[d] & X\ar[d]\\
B\ar[r] & D
}
$$
is $(k-d)$-connected. In particular at least one such square exists if $k\geq d-1$.
\end{cor}

\begin{rem}The existence when $k\geq d-1$ is valid without the hypothesis $k\geq 2$, but the uniqueness when $k\geq d$ fails if $k=1$.
\end{rem}

Next let us work out what truncation has to say about Example
\ref{cocart-replace}. A generalization of the following discussion from squares to high-dimensional cubes will play a key role in section 6.
Suppose that
$$
\xymatrix{
A\ar[r]\ar[d] & B_1\ar[d]\\
B_2 \ar[r] &B_{12}
}
$$
is a square of spaces such that the maps $B_i\rightarrow B_{12}$ are fibrations. Suppose that integers $k_1\geq 2$ and $k_2\geq 2$ are given, such that the map $B_i\rightarrow B_{12}$ is $k_i$-connected. Assume that $B_{12}$ is path-connected. Because $k_i\geq 2$, the spaces $B_1$, $B_2$, and $B_1\times_{B_{12}}B_2$ are all path-connected and have the same fundamental group as $B_{12}$, say $\Gamma$.

As above, $C_*(-)$ for spaces over $B_{12}$ will denote a complex of projective right $\mathbb Z[\Gamma]$-modules, the complex of singular chains with coefficients in $\mathbb Z[\Gamma]$. Let $K_*$ be the algebraic mapping cone of
$$
C_*(A)\rightarrow \holim (C_*(B_1)\rightarrow C_*(B_{12})\leftarrow C_*(B_2)),
$$
Here the homotopy limit of a diagram $P\rightarrow Q\leftarrow R$ of chain complexes may be defined as having $P_m \times Q_{m-1}\times R_m$ as its $m$-th chain module, with boundary given by
$$
\partial(p,q,r)= (\partial p, f(p)-\partial q +g(r),\partial r).
$$
There is a canonical map
$$
\alpha\:  C_*(A\rightarrow B_1\times_{B_{12}}B_2)\rightarrow K_*,
$$
and we may consider the problem of truncation of $A\rightarrow B_1\times_{B_{12}}B_2$ relative to $\alpha$.

A factorization  $A\rightarrow X\rightarrow B_1\times_{B_{12}}B_2$ is a homology truncation with respect to $\alpha$ if and only if the corresponding square
$$
\xymatrix{
X \ar[r]\ar[d] & B_1\ar[d] \\
B_2\ar[r] &  B_{12}
}
$$
is a homology pushout. If it is a \it good\rm\ homology truncation (i.e. if in addition the square is $2$-cartesian), then the square is a homotopy pushout. Indeed, homology pushout implies homotopy pushout if the square is $2$-cocartesian, and $2$-cocartesian follows by the dual Blakers-Massey theorem \ref{dual} from $1$-cartesian combined with $k_1+k_2\geq 1$.

Thus the good homology truncations are precisely those factorizations such that the corresponding square is both a homotopy pushout and $2$-cartesian. They are also those such that the corresponding square is a homotopy pushout and $\pi_1(X)\rightarrow \Gamma$ is an isomorphism.

To apply the truncation theorem, we need a connectivity for the chain map $\alpha$ and we need a cohomological dimension for the chain complex $K_*$.

The square
$$
\xymatrix{
B_1\times_{B_{12}}B_2  \ar[r]\ar[d] & B_1\ar[d] \\
B_2\ar[r] &  B_{12}
}
$$
is $(k_1+k_2+1)$-cocartesian, by Theorem \ref{dual} again, and therefore homologically $(k_1+k_2+1)$-cocartesian, so that the map $\alpha$ is $(k_1+k_2)$-connected.

The chain complex $K_*$ is the desuspension of the algebraic mapping cone of
$$
\hocolim (C_*(B_1)\leftarrow C_*(A)\rightarrow C_*(B_2))\rightarrow C_*(B_{12}),
$$
so it will be $d$-dimensional if the latter is $(d+1)$-dimensional.

We conclude:

\begin{lem} \label{cocart-replace-lemma} Let
$$
\xymatrix{
A \ar[r]\ar[d] & B_1\ar[d] \\
B_2\ar[r] &  B_{12}
}
$$
be a square diagram of spaces in which the maps $B_i\rightarrow B_{12}$ are $k_i$-connected fibrations with $k_i\geq 2$. Assume that the canonical map
$$
\hocolim (B_1\leftarrow A \rightarrow B_2)\rightarrow B_{12}
$$
is homotopically $(d+1)$-dimensional with $d\geq 2$. Then the portion of $w\cal T(A\rightarrow B_1\times_{B_{12}}B_2)$ corresponding to squares
$$
\xymatrix{
X \ar[r]\ar[d] & B_1\ar[d] \\
B_2\ar[r] &  B_{12}
}
$$
that are both homotopy pushouts and $2$-cartesian is $(k_1+k_2-d-1)$-connected.
\end{lem}

\section{\label{analyticity} Preview of the Analyticity Proof}

We now combine the analysis of Example \ref{ur-completion} in the previous section with the description (Prop. \ \ref{hofiber-pushforward}) of the homotopy fiber of a pushforward functor. Recall that $w\cal T(A\rightarrow B,d)$ denotes the portion of $w\cal T(A\rightarrow B)$ consisting of objects $X=(A\rightarrow X\rightarrow B)$ such that the map $A\rightarrow X$ is homotopically $d$-dimensional. If $A_0\rightarrow A_1$ is a cofibration, so that there is a pushforward functor
$$
w\cal T(A_0\rightarrow *)\rightarrow w\cal T(A_1\rightarrow *),
$$
then this functor carries $w\cal T(A_0\rightarrow *,d)$ into $w\cal T(A_1\rightarrow *,d)$.

\begin{prop} \label{pushforward-we}
If the cofibration $A_0\rightarrow A_1$ is $2$-connected and $d\geq 2$, then $w\cal T(A_0\rightarrow *,d)$ is precisely the preimage of $w\cal T(A_1\rightarrow *,d)$. That is, a map $A_0\rightarrow X$ must be homotopically $d$-dimensional if the induced map $A_1\rightarrow A_1\cup_{A_0}X$ is homotopically $d$-dimensional.
\end{prop}

\proof Since $d\geq 2$, it will be enough (by Prop.\ \ref{coh-htpy}) if $A_0\rightarrow X$ is cohomologically $d$-dimensional.
The map $X\rightarrow A_1\cup_{A_0}X$ is $2$-connected, so every coefficient system on $X$ is pulled back from a system on $A_1\cup_{A_0}X$. Therefore for any coefficient system the group $H^i(A_0\rightarrow X)$ can be identified with a group $H^i(A_1\rightarrow A_1\cup_{A_0}X)$. By assumption the latter vanishes for $i>d$.
\endproof

\begin{prop} \label{pushforward-we-connectivity} If $A_0\rightarrow A_1$ is a $k$-connected cofibration with $k\geq 2$ then for any $d\geq 2$ the functor
$$
w\cal T(A_0\rightarrow *,d)\rightarrow w\cal T(A_1\rightarrow *,d)
$$
is $(k-d+1)$-connected.
\end{prop}

\proof The homotopy fiber of the map with respect to any object $A_1\rightarrow X_1$ of  $w\cal T(A_1\rightarrow *,d)$ is the same as the homotopy fiber of
$$
w\cal T(A_0\rightarrow *)\rightarrow w\cal T(A_1\rightarrow *),
$$
by Prop. \ref{pushforward-we}. Therefore, by Prop. \ref{hofiber-pushforward}, it is weakly equivalent to the portion of $w\cal T(A_0\rightarrow X_1)$ corresponding to homotopy pushout squares
$$
\xymatrix{
A_0\ar[r]\ar[d] & ?\ar[d]\\
 A_1 \ar[r] & X_1
}
$$
By Corollary \ref{trunc_cor}, this is $(k-d)$-connected, making the map $(k-d+1)$-connected.
\endproof

A key step in the proof of our main result will be a multirelative connectivity statement (\lq the analyticity of pushforward\rq), of which
\ref{pushforward-we-connectivity} is the first, or $1$-cube, case. We informally discuss the $2$-cube case now.

\begin{thm} \label{square-pushforwards}
Let
$$
\xymatrix{
A\ar[r]\ar[d] & A_1\ar[d]\\
A_2\ar[r] & A_{12}
}
$$
be a pushout square of cofibrations in which the horizontal and vertical maps are respectively $k_1$-connected and $k_2$-connected. Then the resulting square (of categories and pushforward functors)
$$
\xymatrix {
w\cal T(A\rightarrow *,d)\ar[r]\ar[d] & w\cal T(A_1\rightarrow *,d)\ar[d]\\
w\cal T(A_2\rightarrow *,d)\ar[r] & w\cal T(A_{12}\rightarrow *,d)
}
$$
is $(k_1+k_2-d)$-cartesian as long as $k_1\geq 2$ and $k_2\geq 2$ and $d\geq 2$. \end{thm}

We sketch two proofs.

First argument:  Identify the homotopy fibers of the horizontal maps, using Prop. \ref{hofiber-pushforward} as in the proof of
Prop. \ref{pushforward-we-connectivity}. Then, using \ref{hofiber-pushforward} again, identify the homotopy fiber of the induced map between these. Use the truncation theorem to say how highly connected this fiber of fibers is. We invite the reader to fill in the details.

Second argument: (This is the one that will be carefully worked out and used to handle the general case of $r$-cubes in section 6.) The task is to show that for every point in the homotopy limit of
$$
w\cal T(A_2\rightarrow *,d)\rightarrow w\cal T(A_{12}\rightarrow *,d)\leftarrow w\cal T(A_1\rightarrow *,d)
$$
the homotopy fiber of the map from $w\cal T(A\rightarrow *,d)$ is $(k_1+k_2-d-1)$-connected. This means something like the following: fix spaces and maps
$$
\xymatrix {
A_1\ar[r]\ar[d] &  A_{12}\ar[d] &A_2\ar[l]\ar[d]\\
B_1\ar[r]&  B_{12} &B_2\ar[l]
}
$$
such that the vertical maps are $d$-dimensional and the squares are pushouts, and then look at all possible ways of choosing a space $X$ and associated maps to make a cube
$$
\xymatrix{
(A\rightarrow X)\ar[r]\ar[d] & (A_1\rightarrow B_1)\ar[d]\\
(A_2\rightarrow B_2)\ar[r] & (A_{12}\rightarrow B_{12})
}
$$
in which all squares are pushouts and $\pi_1(X)$ is what it should be. Lemma
\ref{cocart-replace-lemma} predicts that the space of all such choices should be $(k_1+k_2-d-1)$-connected.

\section{\label{sketch} Preview of the Main Proof}

We now indicate how the analyticity statement proves the main result. We consider the simplest case, where the problem is to obtain a connectivity estimate for the inclusion
$$
E^h(P,N-Q)\rightarrow E^h(P,N).
$$
Whereas the corresponding problem for smooth embeddings is more or less trivial, for Poincar\'e embeddings the solution will involve all of the main ideas that are needed for the general case. Recall that the expected connectivity is $n-p-q-1$, and that we are going to settle for $n-p-q-2$.

Let $N$, $P$, and $Q$ be $n$-dimensional Poincar\'e spaces with boundary, and suppose that pieces $\partial_0P$ and $\partial_0Q$ of the boundaries of $P$ and $Q$ are embedded disjointly in the boundary of $N$. More precisely, let $(P;\partial_0P,\partial_1P)$, $(Q;\partial_0Q,\partial_1Q)$, and $(N;\partial_0P\coprod\partial_0Q,\partial_2N)$ be $n$-dimensional Poincar\'e triads.

We have maps (each a pushforward functor followed by a pullback functor)
$$
\xymatrix{
 \cal I_n^h(\partial_1 P\cup\partial_2 N\cup\partial_1Q)\ar[r]\ar[d] & \cal I_n^h(\partial_0 P\cup\partial_2 N\cup\partial_1Q)\ar[d]\\
\cal I_n^h(\partial_1 P\cup\partial_2 N\cup\partial_0Q)\ar[r] & \cal I_n^h(\partial_0 P\cup\partial_2 N\cup\partial_0Q)
}
$$
The homotopy fiber of the lower horizontal map with respect to $N$ is $E^h(P,N)$. Any choice of an embedding of $Q$ in $N$ (consistent with the given embedding of $\partial_0Q$ in $\partial N$), gives a point in $I_n^h(\partial_0 P\cup\partial_2 N\cup\partial_1Q)$ (the \lq complement\rq\ of $Q$ in $N$), and then the homotopy fiber of the upper horizontal map with respect to that point is $E^h(P,N-Q)$. The statement to be proved is that for every such choice the induced map
$$
E^h(P,N-Q)\rightarrow E^h(P,N)
$$
of homotopy fibers is $(n-p-q-2)$-connected. Equivalently, the statement to be proved is that the square is $(n-p-q-2)$-cartesian. It is in this formulation, symmetrical with respect to $P$ and $Q$, that it will proved.

The maps in the square above are obtained from maps
$$
\xymatrix{
 \partial_1 P\cup\partial_2 N\cup\partial_1Q\ar[r] \ar[d]&   P\cup\partial_2 N\cup\partial_1Q &\partial_0 P\cup\partial_2 N\cup\partial_1Q\ar[l]\ar[d]\\
 \partial_1 P\cup\partial_2 N\cup Q&  {} &\partial_0 P\cup\partial_2 N\cup Q\\
\partial_1 P\cup\partial_2 N\cup\partial_0Q\ar[r]\ar[u] & P\cup\partial_2 N\cup\partial_0Q & \partial_0 P\cup\partial_2 N\cup\partial_0Q\ar[l]\ar[u]
}
$$
Each right or downward arrow induces a pushforward functor, while each left or upward arrow induces a pullback. Placing $P\cup\partial_2 N\cup Q$ in the center, we arrive at a \lq subdivision\rq\  of the first square above into four squares:
$$
\xymatrix{
 \cal I_n^h(\partial_1 P\cup\partial_2 N\cup\partial_1Q)\ar[r]\ar[d] &  \tilde{\cal I}_n^h( P\cup\partial_2 N\cup\partial_1Q)\ar[r]\ar[d] &\cal I_n^h(\partial_0 P\cup\partial_2 N\cup\partial_1Q)\ar[d]\\
 \tilde{\cal I}_n^h(\partial_1 P\cup\partial_2 N\cup Q)\ar[r]\ar[d] &  \tilde{\cal I}_n^h( P\cup\partial_2 N\cup Q)\ar[r]\ar[d] &\tilde{\cal I}_n^h(\partial_0 P\cup\partial_2 N\cup Q)\ar[d]\\
\cal I_n^h(\partial_1 P\cup\partial_2 N\cup\partial_0Q)\ar[r] & \tilde{\cal I}_n^h(P\cup\partial_2 N\cup\partial_0Q)\ar[r] &\cal I_n^h(\partial_0 P\cup\partial_2 N\cup\partial_0Q)
}
$$
The $\tilde{\cal I}_n^h$ notation is as in Example \ref{example-lift-extend}. For example, $ \tilde{\cal I}_n^h(P\cup\partial_2 N\cup \partial_1Q)$ consists of spaces under  $P\cup\partial_2 N\cup \partial_1Q$ which satisfy $n$-dimensional duality relative to $\partial_0 P\cup\partial_2 N\cup\partial_1Q$, and $ \tilde{\cal I}_n^h(P\cup\partial_2 N\cup Q)$ consists of spaces under  $P\cup\partial_2 N\cup Q$ which satisfy $n$-dimensional duality relative to $\partial_0 P\cup\partial_2 N\cup\partial_0Q$.

For the big square to be be $(n-p-q-2)$-cartesian, it will suffice if all four of the small squares are $(n-p-q-2)$-cartesian. (The proof in the general case (when instead of $Q$  there are $Q_1,\dots ,Q_r$) will use an $(r+1)$-cube subdivided into $2^{r+1}$ cubes.)

The lower right square is $\infty$-cartesian by Prop.
\ref{pullback-functor-diagram}.
Or rather, \ref{pullback-functor-diagram} gives that the square
$$
\xymatrix{
w\mathcal T( P\cup\partial_2 N\cup Q\rightarrow *)\ar[r]\ar[d] &w\mathcal T(\partial_0 P\cup\partial_2 N\cup Q\rightarrow *)\ar[d]\\
w\mathcal T( P\cup\partial_2 N\cup \partial_0Q\rightarrow *)\ar[r] &w\mathcal T(\partial_0 P\cup\partial_2 N\cup \partial_0Q\rightarrow *)
}
$$
is $\infty$-cartesian. To conclude that the lower right square, which is made up of portions of these four categories, is also $\infty$-cartesian, we must be sure (Prop.\ \ref{zero-one}) that for a point of $w\mathcal T( P\cup\partial_2 N\cup Q\rightarrow *)$ to belong to $ \tilde{\mathcal I}_n^h( P\cup\partial_2 N\cup Q)$ it suffices if its image in $w\mathcal T( \partial_0P\cup\partial_2 N\cup Q\rightarrow *)$ belongs to $\tilde{\cal I}_n^h(\partial_0 P\cup\partial_2 N\cup Q)$ and its image in $w\mathcal T( P\cup\partial_2 N\cup \partial_0Q\rightarrow *)$ belongs to$\tilde{\cal I}_n^h(P\cup\partial_2 N\cup\partial_0Q)$. In fact, by definition either one alone suffices.

The upper right square is $(n-p-q-1)$-cartesian by
Prop. \ref{pullback-pushforward-diagram}, with $d=p$ and $k=n-q-1$. Again, more precisely \ref{pullback-pushforward-diagram} shows that
$$
\xymatrix{
w\mathcal T( P\cup\partial_2 N\cup\partial_1Q\rightarrow *)\ar[r]\ar[d] &w\mathcal T(\partial_0 P\cup\partial_2 N\cup\partial_1Q\rightarrow *)\ar[d]\\
w\mathcal T( P\cup\partial_2 N\cup Q\rightarrow *)\ar[r] &w\mathcal T(\partial_0 P\cup\partial_2 N\cup Q\rightarrow *)
}
$$
is $(n-p-q-1)$-cartesian. To see that the same holds for the upper right square, we use \ref{zero-one} again, noting that a point of $w\mathcal T( P\cup\partial_2 N\cup \partial_1Q\rightarrow *)$ belongs to $ \tilde{\mathcal I}_n^h(P\cup\partial_2 N\cup Q)$ if it maps into ${\cal I}_n^h(\partial_0 P\cup\partial_2 N\cup \partial_1Q)$.

The lower left square is just like the upper right but with $P$ and $Q$ reversed, so it is also $(n-q-p-1)$-cartesian.

The upper left square is where the analyticity of pushforward is used. We have a pushout square of cofibrations
$$
\xymatrix{
\partial_1P\cup\partial_2N\cup\partial_2Q  \ar[r]\ar[d] & P\cup\partial_2N\cup\partial_2Q\ar[d]\\
\partial_1P\cup\partial_2N\cup Q  \ar[r] & P\cup\partial_2N\cup Q
}
$$
in which the horizontal and vertical maps are respectively $k_1$-connected and $k_2$-connected with $k_1=n-p-1\geq 2$ and $k_2=n-q-1\geq 2$. Abbreviate this last square as
$$
\xymatrix{
A \ar[r]\ar[d] & A_1\ar[d]\\
A_2 \ar[r] & A_{12}\, .
}
$$
Since a Poincar\'e pair of formal dimension $n$
is homotopically $n$-dimensional, the four
categories in that upper left square
$$
\xymatrix{
\cal I_n^h(A) \ar[r]\ar[d] & \tilde{\cal I}_n^h(A_1)\ar[d]\\
\tilde{\cal I}_n^h(A_2) \ar[r] & \tilde{\cal I}_n^h(A_{12})
}
$$
are portions of the categories in the square:
$$
\xymatrix{
w\cal T(A\rightarrow *,n)\ar[r]\ar[d] & w\cal T(A_1\rightarrow *,n)\ar[d]\\
w\cal T(A_2\rightarrow *,n)\ar[r] & w\cal T(A_{12}\rightarrow *,n)\, .
}
$$
The latter square is $k$-cartesian where $k=k_1+k_2-n=n-p-q-2$, by Theorem
\ref{square-pushforwards}.
Once more we check that Prop. \ref{zero-one} applies.
For an object of $T(A\rightarrow *,n)$ it is not the case that  Poincar\'e duality is implied by Poincar\'e duality after pushforward along $\partial_1P\rightarrow P$ or by Poincar\'e duality after pushforward along $\partial_1Q\rightarrow Q$, but it is implied (Prop.\ \ref{poincare-gluing}) by the conjunction of these.

\section{\label{precise anal} The Analyticity Proof}

We carefully state and prove the generalization of Theorem
\ref{square-pushforwards} to cubes of any dimension.

Let $A_{\bullet}$ be a pushout cube of cofibrations. In other words, let
$$
\lbrace A_{\emptyset}\rightarrow A_i| 1\leq i\leq r\rbrace
$$
be cofibrations and let $A_S$ be the union along $A_{\emptyset}$ of the $A_i$ for $i\in S$.

To each $S\subset \underline r$ is associated the category $w\cal T(A_S\rightarrow *)$, and when $S\subset T\subset \underline r$ there is a pushout functor $w\cal T(A_S\rightarrow *)\rightarrow w\cal T(A_T\rightarrow *)$. This does not precisely amount to a functor
$$
S\mapsto w\cal T(A_S\rightarrow *),
$$
because composition is not strictly preserved. We rectify this, introducing a genuine functor
$$
S\mapsto R_S
$$
such that $R_S$ is categorically equivalent to $w\cal T(A_S\rightarrow *)$ for each $S$.

Denote by $A_{S\subset \bullet}$ the \lq face\rq\ of the cube $A_{\bullet}$ which is the restriction to the poset of subsets of $\underline r$ containing $S$. Let an object (briefly denoted $X$) of $R_S$ be a map
$$
A_{S\subset \bullet}\rightarrow X_{S\subset \bullet}
$$
of $(\underline r-S)$-cubes such that for every choice of $T$ and $U$ with $S\subset T\subset U\subset \underline r$ the square
$$
\xymatrix{
A_T\ar[r]\ar[d] &  X_T\ar[d]\\
A_U\ar[r] & X_U
}
$$
is a categorical pushout. A morphism between two such objects $X$ and $Y$ is a collection of weak equivalences $X_T\rightarrow Y_T$ respecting all the structure, in other words an (objectwise) equivalence of cubes under $A_{S\subset \bullet}$. There is a forgetful functor $R_S\rightarrow R_T$ when $S\subset T$, and this makes $R_S$ a functor of $S$.  A forgetful functor taking $X$ to $X_S$ gives an equivalence of categories from $R_S$ to $wT(A_S\rightarrow *)$. The diagram
$$
\xymatrix{
R_S\ar[r]\ar[d] &  w\cal T(A_S\rightarrow *)\ar[d]\\
R_T\ar[r] & w\cal T(A_T\rightarrow *)
}
$$
commutes up to canonical isomorphism.

Given an integer $d$, there is also the subfunctor $R_{\bullet}(d)$, defined by imposing the restriction \lq homotopical dimension $\leq d$\rq\ on all maps $A_T\rightarrow X_T$.

\begin{thm}[\lq Analyticity of Pushforward\rq]
\label{pushforward-analyticity} Let $A_{\bullet}$ be a pushout cube of cofibrations and assume that for each $i$ the cofibration $A_{\emptyset}\rightarrow A_i$ is $k_i$-connected. Assume also that $d\geq 2$ and that $k_i\geq 2$ for all $i$. Then the cube of categories
$$
S\mapsto R_S(d)\sim w\cal T(A_S\rightarrow *,d)
$$
defined above is $(2-r+k_1+\dots +k_r-d)$-cartesian.
\end{thm}

\proof We first introduce two other cubes related to $R_{\bullet}$ by maps
$$
R_{\bullet}\rightarrow R_{\bullet}^h\leftarrow R_{\bullet}^f.
$$
The category $R^h_S$ is defined just like $R_S$ except that the condition on objects is relaxed: instead of requiring each square
$$
\xymatrix{
A_T\ar[r]\ar[d] &  X_T\ar[d]\\
A_U\ar[r] & X_U
}
$$ to be a categorical pushout, we require it to be a homotopy
pushout. Recall that this means that it strictly commutes and induces
a weak equivalence
$\hocolim (A_U\leftarrow A_T\rightarrow X_T)\rightarrow X_U$.
The morphisms in $R^h_S$ are again defined to be
the objectwise equivalences. Again $R^h_S$ depends functorially on
$S$. The inclusion $R_S\rightarrow R^h_S$ is natural in $S$, and for
every $S$ it is a weak equivalence (i.e., induces a weak equivalence
of nerves), by an argument using functorial factorization just as in
the proof of \ref{replacements}. The full subcategory $R^f_S\subset
R^h_S$ is defined by imposing an additional condition on objects, namely that the
cube $X_{S\subset \bullet}$ is injectively fibrant. That is, for every
$T$ the canonical map
$$
X_T\rightarrow \lim_{T\subsetneq U}X_U
$$
is required to be a fibration. Again the construction is functorial in $S$. Again the inclusion $R^f_S\rightarrow R^h_S$ is natural and is a weak equivalence. All of this remains true if the classes of objects are restricted by imposing the condition of homotopical $d$-dimensionality on the maps $A_T\rightarrow X_T$. Thus we have weak equivalences of cubes
$$
R_{\bullet}(d)\rightarrow R_{\bullet}^h(d)\leftarrow R_{\bullet}^f(d).
$$

We have to show that $R_{\bullet}^f(d)$ is $k$-cartesian where $k=2-r+k_1+\dots +k_r-d$, in other words that the composed map of simplicial sets
$$
NR^f_{\emptyset}(d)\rightarrow \lim_{\emptyset\neq S}NR^f_S(d)
\rightarrow \holim_{\emptyset\neq S} NR^f_S(d)
$$
is $k$-connected, in other words that its homotopy fibers are $(k-1)$-connected. The rest of the proof can be outlined as follows:
\begin{itemize}
\item {\it Step One: }Show that the second map above (lim to holim) is a weak equivalence, so that the question becomes one about the homotopy fibers of the first map.

\item {\it Step Two: }Show that the first map above is a quasifibration (in a sense to be explained below), so that the question becomes one about the strict fibers of the map.

\item {\it Step Three: }Determine the connectivity of these strict fibers, which are nerves of certain categories, by means of the truncation theorem.
\end{itemize}

We take the steps in reverse order.

Step Three: Fix an object of $\lim_{\emptyset\neq S}R^f_S(d)$. This consists precisely of a functor
$$
S\mapsto B_S
$$
from the poset of nonempty subsets of $\underline r$ to spaces, plus a natural map $A_S\rightarrow B_S$, such that (1) for every $T$ and $U$ with $\emptyset\neq T\subset U\subset \underline r$ the square
$$
\xymatrix{
A_T\ar[r]\ar[d] &  B_T\ar[d]\\
A_U\ar[r] & B_U
}
$$
is a homotopy pushout, (2) for every $S$, the homotopical dimension of $B_S$ relative to $A_S$ is at most $d$, and (3) for every nonempty $T$ the canonical map
$$
B_T\rightarrow \lim_{T\subsetneq U}B_U
$$
is a fibration.

An object of $R^f_{\emptyset}$ consists of all of that data plus a space $X$ and maps
$$
A_{\emptyset}\rightarrow X\rightarrow \lim_{S\neq \emptyset}B_S
$$
such that the diagram commutes
$$
\xymatrix{
A_{\emptyset}\ar[r]\ar[d] &  X\ar[d]\\
\lim_{S\neq \emptyset}A_S\ar[r] & \lim_{S\neq \emptyset}B_S
}
$$
and such that for every nonempty $U$ (or equivalently for every singleton $U=\lbrace s\rbrace$) the square
$$
\xymatrix{
A_{\emptyset}\ar[r]\ar[d] &  X\ar[d]\\
A_U\ar[r] & B_U
}
$$
is a homotopy pushout, and such that the map $X\rightarrow \lim_{S\neq \emptyset}B_S$ is a fibration. Thus the strict fiber of the map $NR^f_{\emptyset}(d)\rightarrow \lim_{\emptyset\neq S}NR^f_S(d)$ over a vertex is isomorphic to a portion of the category ${w\cal T^f(A_{\emptyset}\rightarrow
\lim_{S\neq \emptyset}B_S)}$.

Using the assumptions that $d\geq 2$ and $k_i\geq 2$, we can identify the portion in question as (the fibrant objects of) a truncation category. The argument is much as in the proof of \ref{cocart-replace-lemma}. Each map $B_S\rightarrow B_T$ is $2$-connected, being related by a homotopy pushout square to a $2$-connected map $A_S\rightarrow A_T$. We are free to assume that $B_{\underline r}$ is path-connected. Let $\Gamma$ be its fundamental group. Every $B_S$ is then also path-connected with fundamental group $\Gamma$.

Higher Blakers-Massey arguments applied to the $\infty$-cartesian cube formed by the spaces $B_S$ and the space $ \lim_{S\neq \emptyset}B_S$ show that $ \lim_{S\neq \emptyset}B_S$ is also path-connected, with the same fundamental group. We give details in the case $r=3$ only. The square
$$
\xymatrix{
B_3\ar[r]\ar[d] & B_{13}\ar[d]\\
 B_{23} \ar[r] & B_{123}
 }
$$
is $\infty$-cocartesian and the maps $B_3\rightarrow B_{13}$ and $B_3\rightarrow B_{23}$ are respectively $k_1$- and $k_2$-connected, so the square is $(k_1+k_2-1)$-cartesian, in particular $2$-cartesian. The parallel square
$$
\xymatrix{
\lim_{S\neq \emptyset} B_S\ar[r]\ar[d] & B_{1}\ar[d]\\
 B_{2} \ar[r] & B_{12}
 }
$$
is therefore also $2$-cartesian, since the $3$-cube formed by both squares is $\infty$-cartesian. The map $B_2 \rightarrow B_{12}$ is $k_1$-connected, therefore $2$-connected, and it follows that the parallel map $\lim_{S\neq \emptyset}B_S \rightarrow B_1$ is also $2$-connected.

Let $C_*(-)$ denote chains with coefficients in $\mathbb Z\Gamma$. There is an obvious map
$$
\beta\: C_*(\lim_{S\neq \emptyset}B_S)\rightarrow
\holim_{S\neq\emptyset} C_*(B_S)
$$
Let $\alpha$ be the resulting map
$$
C_*(A\rightarrow \lim_{S\neq \emptyset}B_S)\rightarrow K_*,
$$
where $K_*$ is the algebraic mapping cone of
$$
C_*(A)\rightarrow \holim_{S\neq\emptyset} C_*(B_S).
$$

The canonical map from $K_*$ to $C_*(A_s\rightarrow B_s)$ is an equivalence, for each $s\in \underline r$. Indeed, since $A_{\bullet}$ is $\infty$-cocartesian we have an equivalence
$$
C_*(A)\rightarrow \holim_{S\neq\emptyset} C_*(A_S)
$$
and thus an equivalence
$$
K_*\rightarrow \holim_{S\neq\emptyset} C_*(A_S\rightarrow B_S),
$$
while for every nonempty $T\subset  \underline r$ we also have an equivalence
$$
\holim_{S\neq\emptyset} C_*(A_S\rightarrow B_S)\rightarrow C_*(A_T\rightarrow B_T).
$$
This last is a consequence of the fact that,
by (1), the map
$$
C_*(A_T\rightarrow B_U)\rightarrow C_*(A_T\rightarrow B_U)
$$
is an equivalence for every $\emptyset\neq T\subset U\subset \lbrace r\rbrace$.

Therefore, if a factorization
$$
A\rightarrow X\rightarrow \lim_{S\neq\emptyset} B_S
$$
is a truncation relative to $\alpha$, then for each $s$ the resulting square
$$
\xymatrix{
A\ar[r]\ar[d] & X\ar[d]\\
 A_s\ar[r] & B_s
 }
$$
is a homology pushout.

If the truncation is good, then the map $X\rightarrow B_s$ is $2$-connected and therefore the square is a homotopy pushout. In other words, every good truncation corresponds to an object of $R^h_{\emptyset}$.

Conversely, if a factorization corresponds to an object of $R^h_{\emptyset}$ then (by using any $s$ and reversing the argument above) it follows that it is a good homology truncation.

To complete Step Three we need the cohomological dimension of $K_*$ and the connectivity of $\alpha$. The former is $d$. The latter is equal to the connectivity of the map $\beta$. This comes down to the question, how cocartesian is the $r$-cube formed by $\lim_{S\neq \emptyset}B_S$ and the various spaces $B_S$? We use the higher Blakers-Massey theorems. The cube is $\infty$-cartesian. For a proper nonempty subset $T$ the back $T$-face is strongly cocartesian and therefore $(1+ \Sigma_{s\in T} (k_s-1))$-cartesian. The sum of these numbers over a partition is minimized by a two-part partition and this minimum value is $(2-r+\Sigma k_s)$. Therefore the cube is ${(1+\Sigma k_s)}$-cocartesian. It follows that $\beta$ is ${(2-r+\Sigma k_s)}$-connected, $\alpha$ is ${(2-r+\Sigma k_s)}$-connected, and by the truncation theorem the fiber is ${(1-r+\Sigma k_s-d)}$-connected.

Step Two: A map $X\rightarrow Y$ of spaces is called a quasifibration if for every point in $Y$ the canonical map from fiber to homotopy fiber is a weak equivalence. Let us call a map $X\rightarrow Y$ of simplicial sets a quasifibration if it satisfies the following condition (which implies that after realization it is a quasifibration of spaces): For every morphism of simplices over $Y$, i.e. for every diagram of simplicial sets of the form
$$
\Delta^l\rightarrow \Delta^k\rightarrow Y,
$$
the associated map
$$
X\times_Y\Delta^l\rightarrow X\times_Y\Delta^k
$$
is a weak equivalence. Note that this holds in all cases if it holds in the case when $l=0$ and $\Delta^0\rightarrow \Delta^k$ is the first or last vertex map, because in the category of standard simplices all maps become invertible as soon as one inverts all first and last vertex maps. (Proof: Every map $\Delta^l\rightarrow \Delta^k$ may be composed with a first vertex map to get a map $\Delta^0\rightarrow \Delta^k$. Every map $\Delta^0\rightarrow \Delta^k$ may be expressed as a last vertex map $\Delta^0\rightarrow \Delta^m$ followed by a map $\Delta^m\rightarrow \Delta^k$ whose composition with a first vertex map is a first vertex map.)

Clearly in the category of simplicial sets any map obtained by pullback from a
quasifibration is a quasifibration. Moreover, if $X'\subset X$ is a union of components of $X$ and $X\rightarrow Y$ is a quasifibration then $X'\rightarrow Y$ is a quasifibration.

Consider the functor
$$
R^f_{\emptyset}(d)\rightarrow \lim_{\emptyset\neq S}R^f_S(d).
$$
It is part of a square
$$
\xymatrix{
R^f_{\emptyset}(d)\ar[r]\ar[d] & w\Arr^f \cal T(A_{\emptyset}\rightarrow *)\ar[d]\\
\lim_{\emptyset\neq S}R^f_S(d)\ar[r] & w\cal T(A_{\emptyset}\rightarrow *)
}
$$
which we will describe shortly. This is not a pullback square, but it embeds $R^f_{\emptyset}(d)$ as a union of components of the fiber product. Therefore it will suffice if the right-hand arrow is a quasifibration.

Here is the square: The lower arrow is the functor taking $A_{\bullet}\rightarrow B_{\bullet}$ to the composed map
$$
A_{\emptyset}\rightarrow \lim_{S\neq \emptyset}A_S\rightarrow \lim_{S\neq \emptyset}B_S;
$$
here we have used the fact that the diagrams $B_{\bullet}$ are injectively fibrant to ensure that this really takes weak equivalences to weak equivalences. We define the category $w\Arr^f \cal T(A_{\emptyset}\rightarrow *)$. Its objects are the diagrams
$$
A_{\emptyset}\rightarrow X\rightarrow Y
$$
such that $X\rightarrow Y$ is a fibration. Its morphisms are the maps of diagrams that are weak equivalences at $X$ and $Y$ (and identity at $A_{\emptyset}$). The right-hand arrow in the square takes $A_{\emptyset}\rightarrow X\rightarrow Y$ to the composition $A_{\emptyset}\rightarrow Y$. The upper arrow takes $A_{\bullet}\rightarrow B_{\bullet}$ to the composition $A_{\emptyset}\rightarrow B_{\emptyset}\rightarrow  \lim_{S\neq \emptyset}B_S$; again the fibrancy condition is being used.

We now argue that the right-hand arrow is a quasifibration. Take any $p$-simplex in the nerve of $w\cal T(A_{\emptyset}\rightarrow *)$, in other words any diagram
$$
Y_0\rightarrow \dots\rightarrow Y_p
$$
of weak equivalences of spaces under $A_{\emptyset}$. The fiber product
$$
Nw\Arr^f \cal T(A_{\emptyset}\rightarrow *)\times_{Nw\cal T(A_{\emptyset}\rightarrow *)}\Delta^p
$$
 over this simplex is the nerve of a category.  An object is a pair $(i,X\rightarrow Y_i)$, where $0\leq i\leq p$ and $X\rightarrow Y_i$ is a fibration of spaces under $A_{\emptyset}$, and a morphism $(i,X\rightarrow Y_i)\rightarrow (i',X'\rightarrow Y_{i'})$ (necessarily with $i\leq i'$) is a weak equivalence $X\rightarrow X'$ making the square
$$
\xymatrix{
X\ar[r]\ar[d] & X'\ar[d]\\
Y_i\ar[r] & Y_{i'}
}
$$
commute.

We need a functor from the category just described to the fiber over the first vertex (i.e., to $w\cal T^f(A_{\emptyset}\rightarrow Y_0)$), and we need this functor to be an inverse, up to weak equivalence, to the inclusion. We also need another such functor for the last vertex. For the last vertex, send $X\rightarrow Y_i$ to a fibrant replacement for the composition $X\rightarrow Y_i\rightarrow Y_p$. For the first, send $X\rightarrow Y_i$ to $X\times_{Y_i}Y_0\rightarrow Y_0$.

Step One: This uses the work done in Step Two. We have to compare a limit with a homotopy limit. The limit can be described as the limit (fiber product) of the following diagram of three limits
$$
NR^f_{\lbrace r\rbrace}(d)\rightarrow \lim_{\emptyset\neq S}NR^f_{S\cup\lbrace r\rbrace}(d)\leftarrow  \lim_{\emptyset\neq S}NR^f_ S(d),
$$
where $S$ now ranges over subsets of $\lbrace 1,\dots , r-1\rbrace$. The homotopy limit is likewise the homotopy limit of a diagram of homotopy limits
$$
NR^f_{\lbrace r\rbrace}(d)\rightarrow
\holim_{\emptyset\neq S} NR^f_{S\cup\lbrace r\rbrace}(d)\leftarrow
 \holim_{\emptyset\neq S} NR^f_ S(d).
$$
The left arrow in the diagram of limits above is a quasifibration,
by an argument essentially identical to the argument used in Step
Two. It follows that the limit of the limits is equivalent to the
homotopy limit of limits. It remains only to show that the canonical
map from the diagram of three limits to the diagram of three homotopy
limits is an equivalence in all three places, for then it will follow
that it induces an equivalence of homotopy limits. On the right side
the map is an equivalence by induction on $r$. On the left side it is
an identity map. In the middle it is an equivalence, essentially by
induction on $r$ again; there is an equivalence, natural in $S$, from
$R^f_{S\cup\lbrace r\rbrace}(d)$ to what would be $R^f_{S}(d)$ if the
original $r$-cube $A_{\bullet}$ were replaced by one of its
$(r-1)$-dimensional faces.  \endproof

\section{\label{Proof} The Main Proof}

We prove Theorem \ref{disjunction} by a straightforward generalization of the argument of section 5. Instead of $P$ and $Q$ we have $P$ and $Q_1,\dots,Q_r$, and instead of breaking up a square into four squares we break up an $(r+1)$-cube into $2^{r+1}$ cubes. Since the approach is to treat $P$ the same as the $Q_i$, we simply call it $Q_{r+1}$ and let $s=r+1\geq 2$

Assume that for each $i\in \underline s=\lbrace 1,\dots ,s\rbrace$ we have an $n$-dimensional Poincare triad $(Q_i;\partial_0Q_i,\partial_1Q_i)$. Let $(N;\partial_0N,\partial_2N)$ be another $n$-dimensional Poincare triad such that $\partial_0N$ is the disjoint union of the $\partial_0Q_i$.

In fact $N$ will play no role for a while. We just use the Poincar\'e triads $Q_i$ and the Poincar\'e pair $(\partial_2N,\amalg_i\partial\partial_0Q_i)$.

In section 5 we drew two-dimensional diagrams. Now we need some notation because we cannot draw $s$-dimensional diagrams. Maybe the quickest way to convey the notation is to say that in the case $s=2$ the diagram
$$
\xymatrix{
 \partial_1 P\cup\partial_2 N\cup\partial_1Q\ar[r] \ar[d]&   P\cup\partial_2 N\cup\partial_1Q\ar[d] &\partial_0 P\cup\partial_2 N\cup\partial_1Q\ar[l]\ar[d]\\
 \partial_1 P\cup\partial_2 N\cup Q\ar[r] &  P\cup\partial_2 N\cup Q &\partial_0 P\cup\partial_2 N\cup Q\ar[l] \\
\partial_1 P\cup\partial_2 N\cup\partial_0Q\ar[r]\ar[u] & P\cup\partial_2 N\cup\partial_0Q\ar[u] & \partial_0 P\cup\partial_2 N\cup\partial_0Q\ar[l]\ar[u]
}
$$
from section 5 would now be written
$$
\xymatrix{
A_\emptyset^\emptyset\ar[r] \ar[d]&   A_\emptyset^{\lbrace 1\rbrace}\ar[d] & A_{\lbrace 1\rbrace}^{\lbrace 1\rbrace}\ar[l]\ar[d]\\
 A_\emptyset^{\lbrace 2\rbrace}\ar[r] &  A_\emptyset^{\lbrace 1,2\rbrace} &A_{\lbrace 1\rbrace}^{\lbrace 1,2\rbrace}\ar[l] \\
A_{\lbrace 2\rbrace}^{\lbrace 2\rbrace}\ar[r]\ar[u] & A_{\lbrace 2\rbrace}^{\lbrace 1,2\rbrace}\ar[u] & A_{\lbrace 1,2\rbrace}^{\lbrace 1,2\rbrace}\ar[l]\ar[u] \, .
}
$$
and that the resulting diagram of categories
$$
\xymatrix{
 \cal I_n^h(\partial_1 P\cup\partial_2 N\cup\partial_1Q)\ar[r]\ar[d] &  \tilde{\cal I}_n^h( P\cup\partial_2 N\cup\partial_1Q)\ar[r]\ar[d] &\cal I_n^h(\partial_0 P\cup\partial_2 N\cup\partial_1Q)\ar[d]\\
 \tilde{\cal I}_n^h(\partial_1 P\cup\partial_2 N\cup Q)\ar[r]\ar[d] &  \tilde{\cal I}_n^h( P\cup\partial_2 N\cup Q)\ar[r]\ar[d] &\tilde{\cal I}_n^h(\partial_0 P\cup\partial_2 N\cup Q)\ar[d]\\
\cal I_n^h(\partial_1 P\cup\partial_2 N\cup\partial_0Q)\ar[r] & \tilde{\cal I}_n^h(P\cup\partial_2 N\cup\partial_0Q)\ar[r] &\cal I_n^h(\partial_0 P\cup\partial_2 N\cup\partial_0Q)
}
$$
or rather a rectified version that strictly commutes, would be written
$$
\xymatrix{
\cal I_\emptyset^\emptyset\ar[r] \ar[d]&   \cal I_\emptyset^{\lbrace 1\rbrace}\ar[d]\ar[r] & \cal I_{\lbrace 1\rbrace}^{\lbrace 1\rbrace}\ar[d]\\
 \cal I_\emptyset^{\lbrace 2\rbrace}\ar[r]\ar[d] &  \cal I_\emptyset^{\lbrace 1,2\rbrace}\ar[r]\ar[d] &\cal I_{\lbrace 1\rbrace}^{\lbrace 1,2\rbrace}\ar[d] \\
\cal I_{\lbrace 2\rbrace}^{\lbrace 2\rbrace}\ar[r] & \cal I_{\lbrace 2\rbrace}^{\lbrace 1,2\rbrace}\ar[r] & \cal I_{\lbrace 1,2\rbrace}^{\lbrace 1,2\rbrace} \, .
}
$$

The general definitions are as follows: For $S\subset T\subset \underline s$ let $A_S^T$ be the pushout of
$$
\partial_2 N\leftarrow \partial\partial_2 N=\coprod_i\partial_0Q_i\rightarrow \coprod_i Q_i(S,T)
$$
where $Q_i(S,T)=\partial_0Q_i$ if $i\in S$, $Q_i(S,T)=Q_i$ if $i\in T-S$, and $Q_i(S,T)=\partial_1Q_i$ if $i\in \underline s -T$.

This space depends functorially on $(S,T)$, covariantly in $T$ and contravariantly in $S$. In other words, if the pairs are ordered by letting $(S,T)\leq(S',T')$ mean $S'\subset S$ and $T\subset T'$ then we have a diagram of spaces indexed by that poset. Inside this diagram of spaces is a cubical diagram for each subset $S\subset \underline s$, consisting of the spaces $A_{S\cap W}^{S\cup W}$. It might be denoted $A_{S\cap \bullet}^{S\cup \bullet}$. It is a pushout cube of cofibrations. The spaces $A_S^S$ are at the \lq corners\rq\ of the big diagram and at the \lq initial corners\rq\ of the little cubes. For each $S$ the space $A_S^S$ satisfies $(n-1)$-dimensional duality, as it is made from $\partial_2N$ by gluing in $\partial_0Q_i$ for $i\in S$ and gluing in $\partial_1Q_i$ for $i\in\underline s - S$ along the appropriate parts of the boundary.

Let $\cal I_S^T$ be the following category: An object $X^\bullet$ is an $(\underline s-T)$-cube under $A_S^{\bullet}$ satisfying two conditions. Thus it consists of (1) spaces $X^U$ for all $U$ such that $T\subset U\subset\underline s$, (2) maps between these making them a covariant functor of $U$, and (3) natural maps $A_S^U\rightarrow X^U$. The additional conditions are that the square
$$
\xymatrix{
A_S^U\ar[r]\ar[d] & X^U\ar[d]\\
A_S^V\ar[r] & X^V}
$$
should be a pushout for every $U$ and $V$ with $T\subset U\subset V$, and that for every $U$ the composed map
$$
A_U^U\rightarrow A_S^U\rightarrow X^U
$$
should satisfy $n$-dimensional relative Poincar\'e duality. A morphism from $X^\bullet$ to $Y^\bullet$ is a collection of weak equivalences $X^U\rightarrow Y^U$ natural in $U$ and compatible with the maps from $A_S^U$.

Because the squares are pushouts, the duality for general $U$ follows from the duality for $U=T$, by Prop. \ref{gluing_pd}. Again because the squares are pushouts, the forgetful functor $\cal I_S^T\rightarrow w\cal T(A_S^T\rightarrow *)$ that takes $A_S^\bullet\rightarrow X^\bullet$ to $A_S^T\rightarrow X^T$ gives an equivalence of categories to a portion of $w\cal T(A_S^T\rightarrow *)$, namely the category $\tilde {\cal I}_n^h(A_S^T)$ of objects that satisfy duality relative to $A_T^T$.

The category $\cal I_S^T$ depends functorially on $(S,T)$, this time covariantly in both variables. In passing from $(S,T)$ to $(S',T')$ where $S\subset S'$ and $T\subset T'$, one simply discards those spaces $X^U$ such that $U$ does not contain $T'$, while for the remaining choices of $U$ one replaces the map $A_S^U\rightarrow X^U$ by its composition with $A_{S'}^T\rightarrow A_S^T$.

The indexing category for this diagram, the poset of pairs $(S,T)$, is the product of $s$ copies of an ordered set having three elements. We think of the diagram as a big $s$-cube subdivided into $2^s$ little $s$-cubes.

The objects in the big cube are the categories $\cal I_S^S$, which are categorically equivalent to the $\cal I_n^h(A_S^S)$ but have the virtue of depending in a strictly functorial way on $S$.

\begin{thm} \label{ISS} Let $s\geq 2$, let $(Q_1,\dots,Q_s)$ be $n$-dimensional Poincar\'e triads $(Q_i;\partial_)Q_i,\partial_1Q_i)$, and let $(\partial_2N,\amalg_i \partial_0Q_i)$ be an $n$-dimensional Poincar\'e pair. Assume that $Q_i$ has homotopy spine dimension $q_i$ relative to $\partial_0Q_i$ and $n-q_i\geq 3$ for every $i$. Then the cubical diagram $\lbrace \cal I_S^S \rbrace$ defined above is $2-n+\Sigma_i(n-q_i-2)$-cartesian.
\end{thm}

We remind the reader what this has to do with Poincar\'e embeddings. If an object $C^\bullet$ is chosen in $\cal I_{\lbrace s\rbrace}^{\lbrace s\rbrace}$, then in effect we have chosen a Poincar\'e pair $(N,\partial N)$ such that $\partial N$ is $A_{\underline s}^{\underline s}$, the union of $\coprod_i\partial_0Q_i$ and $\partial_2N$, as well as a Poincar\'e embedding of the disjoint union $Q_1\cup\dots\cup Q_{s-1}$ in $N$ compatible with the given Poincar\'e embedding of $\partial_0Q_1\cup\dots\cup \partial_0Q_{s-1}$ in $\partial N$, with complement $C^{\lbrace s\rbrace}$. For each $S\subset\underline{s-1}$, the homotopy fiber of
$$
\cal I_S^S\rightarrow \cal I_{S\cup\lbrace s\rbrace}^{S\cup\lbrace s\rbrace}
$$
with respect to the basepoint determined by $C^\bullet$ is equivalent to the space of Poincar\'e embeddings of $Q_{\lbrace s\rbrace}$ in $C^S$, the complement in $N$ of $Q_S$. The conclusion of the theorem states that this (strictly commutative replacement for the) cube of Poincar\'e embedding spaces ${E^h(Q_{\lbrace s\rbrace},N-Q_\bullet)}$ is $2-n+\Sigma_i(n-q_i-2)$-cartesian.

\begin{proof}[Proof of \ref{ISS}] It is enough if each of the little cubes is ${2-n+\Sigma_i(n-q_i-2)}$-cartesian. There is one of these,  $\cal I_{S\cap\bullet}^{S\cup\bullet}$, for each subset $S\subset\underline s$. We will show the following:

Case 1: If $S$ has more than one element then $\cal I_{S\cap\bullet}^{S\cup\bullet}$ is $\infty$-cartesian.

Case 2: If $S$ has one element then $\cal I_{S\cap\bullet}^{S\cup\bullet}$ is $3-n+\Sigma_i(n-q_i-2)$-cartesian.

Case 3: If $S$ is empty then $\cal I_{S\cap\bullet}^{S\cup\bullet}$ is $2-n+\Sigma_i(n-q_i-2)$-cartesian.

For Case 1, choose two elements $a$ and $b$ in $S$. For each subset $T$ of $\underline s-\lbrace a,b\rbrace$, consider the square
$$
\xymatrix{
\cal I_{S\cap T}^{S\cup T} \ar[r]\ar[d] & \cal I_{(S\cap T)\cup\lbrace a\rbrace}^{S\cup T}\ar[d]\\
\cal I_{(S\cap T)\cup\lbrace b\rbrace}^{S\cup T}\ar[r] &  \cal I_{(S\cap T)\cup\lbrace a, b\rbrace}^{S\cup T}
}
$$
This is one of $2^{s-2}$ parallel faces of the little cube, and it is enough if these squares are all $\infty$-cartesian. The square above can be replaced by the equivalent square
\begin{equation}\label{I-tilde}
\xymatrix{
\tilde {\cal I}_n^h(A_{S\cap T}^{S\cup T}) \ar[r]\ar[d] & \tilde {\cal I}_n^h(A_{(S\cap T)\cup\lbrace a\rbrace}^{S\cup T})\ar[d]\\
\tilde {\cal I}_n^h(A_{(S\cap T)\cup\lbrace b\rbrace}^{S\cup T})\ar[r] &  \tilde {\cal I}_n^h(A_{(S\cap T)\cup\lbrace a, b\rbrace}^{S\cup T})
}
\end{equation}
The categories here are portions of those in the square:
$$
\xymatrix{
w\cal T (A_{S\cap T}^{S\cup T}\rightarrow *) \ar[r]\ar[d] & w\cal T (A_{(S\cap T)\cup\lbrace a\rbrace}^{S\cup T}\rightarrow *)\ar[d]\\
w\cal T (A_{(S\cap T)\cup\lbrace b\rbrace}^{S\cup T}\rightarrow *)\ar[r] &  w\cal T (A_{(S\cap T)\cup\lbrace a, b\rbrace}^{S\cup T}\rightarrow *)
}
$$
This square of pullback functors is $\infty$-cartesian by Prop. \ref{pullback-functor-diagram}, since
$$
\xymatrix{
A_{S\cap T}^{S\cup T}  & A_{(S\cap T)\cup\lbrace a\rbrace}^{S\cup T}\ar[l]\\
A_{(S\cap T)\cup\lbrace b\rbrace}^{S\cup T}\ar[u] &  A_{(S\cap T)\cup\lbrace a, b\rbrace}^{S\cup T}\ar[l]\ar[u]
}
$$
is a pushout square of cofibrations. To conclude that the square
\eqref{I-tilde} is  also $\infty$-cartesian we use Prop. \ref{zero-one}.
In fact, an object $(A_{S\cap T}^{S\cup T}\rightarrow X)$ of $w\cal T (A_{S\cap T}^{S\cup T}\rightarrow *)$ must be in $\tilde {\cal I}_n^h(A_{S\cap T}^{S\cup T})$ if its image $(A_{(S\cap T)\cup\lbrace a\rbrace}\rightarrow X)$ in $w\cal T (A_{(S\cap T)\cup\lbrace a\rbrace}^{S\cup T}\rightarrow *)$ is in $\tilde {\cal I}_n^h(A_{(S\cap T)\cup\lbrace a\rbrace}^{S\cup T})$, because in each case the condition for membership in the subcategory is that $X$ satisfies duality relative to $A_{S\cup T}^{S\cup T}$.

For case 2, suppose $S=\lbrace a\rbrace$ and consider, for each $T\subset \underline s-\lbrace a\rbrace$, the map
$$
\cal I_{S\cap T}^{S\cup T} \rightarrow \cal I_{(S\cap T)\cup\lbrace a\rbrace}^{S\cup T}
$$ These are $2^{s-1}$ parallel edges of the little cube. Fix a choice
$X^\bullet$ of object in $\cal I_{\lbrace a\rbrace}^S$ and thus
compatible choices in $\cal I_{(S\cap T)\cup\lbrace a\rbrace}^{S\cup
T}$ for all $T$. Consider the resulting $(\underline s-\lbrace
a\rbrace$)-cube of homotopy fibers. By Prop.
\ref{fiber-cube-connectivity}
it is enough if the latter cube is
$3-n+\Sigma_i(n-q_i-2)$-cartesian. For each $T$, that homotopy fiber
is equivalent to $\Map(Q_a,X^T)$, the space of maps fixed on
$\partial_0 Q_a$, by Prop. \ref{function-space}. The equivalences are
natural in $T$. The cube $X^{S\subset\bullet}$ is a pushout cube of
cofibrations. For each $i\in \underline s-\lbrace a\rbrace$ the map
$X^S\rightarrow X^{S\cup\lbrace i\rbrace}$ is $(n-q_i-1)$-connected
because it is obtained by pushout from $A_S^S\rightarrow
A_S^{S\cup\lbrace i\rbrace}$ and ultimately from
$\partial_1Q_i\rightarrow Q_i$. Therefore the cube is $1+\Sigma_{i\neq
a}(n-q_i-2)$-cartesian by Theorem \ref{Blakers-Massey-n-ad}.  Since $Q_a$ is
$q_a$-dimensional relative to $\partial_0 Q_a$, the cube
$\Map(Q_a,X^T)$ is $k$-cartesian for $k=-q_a+1+\Sigma_{i\neq
a}(n-q_i-2)=3-n+\Sigma_{i\in\underline s}(n-q_i-2)$.

Case 3 is an almost immediate consequence of Theorem \ref{pushforward-analyticity}. Apply the latter with $A_{\bullet}=A_{\empty}^S$ and $d=n$. This cube $\cal I_{\empty}^{\bullet}$ that we have to deal with is closely related to the cube which is called $R_{\bullet}(n)$ after the statement of \ref{pushforward-analyticity}.
In fact, $\cal I_{\empty}^S$ is the portion of $R_S$
defined by a duality condition. As at the end of section 5, we use \ref{poincare-gluing} to see that an object of $R_{\empty}$ must satisfy the duality condition if its image in every $R_{\lbrace i\rbrace}$ does so, and then invoke Prop. \ref{zero-one} to finish the job.
\end{proof}

\section{\label{dim and conn} Appendix A: Dimension and Connectivity}

This is just a review of some conventions and basic facts.

A space $X$ is called $k$-connected if for every integer $j$ with $-1\leq j\leq k$ every map $S^j\rightarrow X$ can be extended to a map $D^{j+1}\rightarrow X$. If $0\leq k$ then this means that $X$ has exactly one path-component and has trivial $\pi_j$ for $1\leq j\leq k$. If $k=-1$ then it means that $X$ is not empty. If $k\leq -2$ then every space is $k$-connected.

A map $f\:X\rightarrow Y$ of spaces is called $k$-connected if for every point in $Y$ the homotopy fiber of $f$ is $(k-1)$-connected. This is the same as saying that for every basepoint in $X$ the induced map $\pi_j(X)\rightarrow \pi_j(Y)$ is surjective for $1\leq j\leq k$ and injective for $1\leq j<k$, and also that $\pi_0(X)\rightarrow \pi_0(Y)$ is surjective if $0\leq k$ and injective if $0<k$.

If $h=f\circ g$ then the following three implications hold: The map $h$ is $k$-connected if both $f$ and $g$ are $k$-connected. The map $f$ is $k$-connected if $h$ is $k$-connected and $g$ is $(k-1)$-connected. The map $g$ is $k$-connected if $h$ is $k$-connected and $f$ is $(k+1)$-connected. In particular, the property of $k$-connectedness of maps is homotopy-invariant in the sense that if the vertical maps in a square diagram are weak equivalences and one of the horizontal maps is $k$-connected then the other horizontal map is also $k$-connected.

We say that a map $f\:X\rightarrow Y$ of spaces is homotopically $d$-dimensional, or that its homotopy dimension is at most $d$, if in the weak homotopy category of spaces under $X$ it is a retract of a relative CW complex of dimension at most $d$. This property of maps is again homotopy-invariant.

A cofibration (for example a relative CW complex) has homotopical dimension $\leq d$ if and only if has the left lifting property with respect to every $d$-connected fibration. A fibration is $k$-connected if and only  if it has the right lifting property with respect to every homotopically $k$-dimensional cofibration (or relative CW complex).

If the cofibration $A\rightarrow X$ is $d$-dimensional and the fibration $Y\rightarrow B$ is $k$-connected then the space of solutions $X\rightarrow Y$ of the lifting/extension problem
$$
\xymatrix {
A \ar[r]\ar[d] & Y\ar[d]\\
X \ar[r] & B
}
$$
is $(k-d-1)$-connected. (In particular when $k\geq d$ the space of solutions is nonempty.) We leave the proof to the reader, except for the following comment: For the most part the proof can be carried out by obstruction theory, using cohomology of $A\rightarrow X$ with coefficients in local systems defined by homotopy groups of the fibers of $Y\rightarrow B$, but special arguments are needed if the fibers have nontrivial $\pi_0$ or $\pi_1$.

Of course, homotopically $d$-dimensional implies cohomologically $d$-dimensional, meaning the vanishing of relative cohomology in dimensions $>d$ for all coefficient systems on the codomain. Conversely, we have

\begin{prop}\label{coh-htpy}If $d\geq 2$ then a cohomologically $d$-dimensional map is homotopically $d$-dimensional.
\end{prop}

\proof We may assume the map is a cofibration $A\rightarrow B$. Factor it $A\rightarrow Y\rightarrow B$ as a homotopically $d$-dimensional cofibration followed by a $d$-connected fibration. A section of the fibration fixed on $A$ will establish $B$ as a retract of $Y$ relative to $A$. To produce the section, note that the fibers of the fibration have trivial $\pi_0$ and $\pi_1$ (because $d\geq 2$), and use obstruction theory.
\endproof

\section{\label{cubes} Appendix B: Cubical Diagrams}

We make conventions concerning cubical diagrams (\lq cubes\rq) and recall the basic results that are needed. Most of this is in \cite{Goodwillie_CALC2}, where more details and proofs can be found, although some things are said a little differently there.

Let $S$ be a set with $n$ elements. Frequently $S$ will be $\underline n =\lbrace 1\dots n\rbrace$. Let $\mathcal P(S)$ be the poset of all subsets of $S$. A functor $\mathcal P(S)\rightarrow \cal T$ is called an $S$-cube, or $n$-cube, or $n$-dimensional cube, of spaces. On occasion we also refer to a functor as an $S$-cube or $n$-cube when its domain is some poset isomorphic to $\mathcal P(S)$, for example $\mathcal P(S)^{op}$. We may informally refer to a cube $X$ by its collection of spaces $\lbrace X(T)|T\subset S\rbrace$ when the maps between them are understood.

An $n$-cube $X$ has various \it faces\rm\ $\partial_U^T X$, which are cubes whose dimensions range from $0$ to $n$. Here $(T,U)$ is a pair of sets satisfying $U\subset T\subset S$ and $\partial_U^T X$ is a $(T-U)$-cube, the restriction of $X$ to the poset of sets containing $U$ and contained in $T$. The most important are the \it front faces\rm\ $\partial^TX=\partial_{\emptyset}^TX$ and the \it back faces\rm\ $\partial_TX=\partial_T^SX$.

Let $\mathcal P_0(S)\subset \mathcal P(S)$ be the subposet of nonempty subsets of $S$. An $n$-cube $X$ determines a map
$$
a(X)\: X(\emptyset )\rightarrow \holim (X|_{\mathcal P_0(S)})
$$
from the \lq first\rq\ space  to the homotopy limit of the restriction to $\mathcal P_0(S)$.  Dually, if $\mathcal P_1(S)$ is the poset of proper subsets of $S$ then we have the canonical map
$$
b(X)\:\hocolim (X|_{\mathcal P_1(S))}\rightarrow X(S).
$$

The $n$-cubes are the objects of a category, with natural maps as the morphisms. A map $X\rightarrow Y$ of cubes is called a weak equivalence if it is an objectwise weak equivalence, in other words if for every $T$ the map $X(T)\rightarrow Y(T)$ is a weak equivalence.

The category of $S$-cubes becomes a model category if we declare that a fibration of cubes is an objectwise fibration. In this \it projective\rm\ model structure a map $X\rightarrow Y$ is a cofibration if and only if for every subset $T$ of $S$ the diagram
$$
\xymatrix{ \colim (X|_{\mathcal P_1(T))}\ar[r]\ar[d] & X(T)\ar[d]\\
\colim (Y|_{\mathcal P_1(T))}\ar[r] &  Y(T)
}
$$
induces a cofibration from the pushout to $Y(T)$.

Recall that an $(r+1)$-ad consists of a space $X$ and subspaces $X_1,\dots ,X_r$. An $(r+1)$-ad determines an $r$-cube of spaces, consisting of the intersections $\cap_{t\in T} X_t$ (with the convention that when $T$ is empty the intersection is $X$). In this paper the $(r+1)$-ads that are used will generally be CW, in the sense that $X$ is a CW complex and each $X_i$ is a subcomplex. Every $r$-cube is equivalent to a CW $(r+1)$-ad. (We could just as well have said \lq cellular\rq\ instead of \lq CW\rq, meaning that cells are not required to be attached in order of dimension and that retracts are allowed. Cellular $(r+1)$-ads correspond precisely to projectively cofibrant $r$-cubes.)

In the alternative \it injective\rm\ model structure the cofibrations are defined objectwise and the fibrations are the maps such that for every $T$ the diagram
$$
\xymatrix{
X(T)\ar[r]\ar[d] & \lim(X|_{\mathcal P_0(S-T))}\ar[d]\\
Y(T)\ar[r]& \lim(Y|_{\mathcal P_0(S-T))}
}
$$
induces a fibration from $X(T)$ to the pullback.

The cube $X$ is called $k$-cartesian if the map $a(X)$ is $k$-connected. It is called $k$-cocartesian if the map $b(X)$ is $k$-connected. Thus for a $1$-cube (= map) $k$-cartesian and $k$-cocartesian both mean $k$-connected, but for $k\geq 2$ they mean two different things. They may be thought of as two different notions of connectivity of a cube.

A cube that is weakly equivalent to a $k$-[co]cartesian cube is again $k$-[co]cartesian.

An $(r+1)$-cube can be regarded as a map $X\rightarrow Y$ of $r$-cubes. The following simple rules (see page 303-305 of \cite{Goodwillie_CALC2}) apply:

\begin{prop} \label{simple_rule_1}
$X$ is $k$-cartesian if both $Y$ and $X\rightarrow Y$ are $k$-cartesian. $X\rightarrow Y$ is $k$-cartesian if $X$ is $k$-cartesian and and $Y$ is $(k+1)$-cartesian.
(Warning: $Y$ need not be $k$-cartesian even if $X$ and $X\rightarrow Y$ are ${\infty}$-cartesian, as shown by trivial examples involving the empty space.)
\end{prop}

\begin{prop}
$Y$ is $k$-cocartesian if both $X$ and $X\rightarrow Y$ are $k$-cocartesian. $X\rightarrow Y$ is $k$-cocartesian if $Y$ is $k$-cocartesian and $X$ is $(k-1)$-cocartesian. (Warning: $X$ need not be $k$-cocartesian even if $Y$ and $X\rightarrow Y$ are ${\infty}$-cocartesian, as shown by simple examples involving nonabelian fundamental groups.)
\end{prop}

Let $X\rightarrow Y$ be a map of $r$-cubes. If $Y$ is made into a cube of based spaces by choosing a basepoint in $Y(\emptyset)$, then the homotopy fibers of the maps $X(T)\rightarrow Y(T)$ constitute  another $r$-cube $Z_y$, which can be called the (objectwise) homotopy fiber. Thus
$$
Z_y(T) = \text{\rm hofiber}(X(T)\rightarrow Y(T)).
$$

\begin{prop} \label{fiber-cube-connectivity}
If $X\rightarrow Y$, considered as an $(r+1)$-cube, is $k$-cartesian, then $Z_y$ is also $k$-cartesian. The converse holds: $X\rightarrow Y$ is $k$-cartesian if $Z_y$ is $k$-cartesian for every $y\in Y(\emptyset)$. Warning: It is important to consider all possible points $y$ (or at least one point in every path-component of $Y(\emptyset)$) in applying this test.
\end{prop}

We single out a trivial but useful  case of \ref{fiber-cube-connectivity}:

\begin{prop} \label{zero-one}
Let $X\rightarrow Y$ be a map of $r$-cubes such that for each $T$ the map $X(T)\rightarrow Y(T)$ is the inclusion of an open and closed subset. Then the associated $(r+1)$-cube is $\infty$-cartesian as long as the following holds: for every point $y\in Y(\emptyset)$, $y$ must belong to $X(\emptyset)$ if for every singleton $T=\lbrace t\rbrace$ the image of $y$ in $Y(T)$ belongs to $X(T)$. In particular, in this case if $Y$ is $k$-cartesian then $X$ is $k$-cartesian.
\end{prop}

\begin{proof} In the cube $Z_y$ of homotopy fibers, each space $Z_y(T)$ is either empty or weakly contractible. If $Z_y(\emptyset)$ is nonempty then of course the others are, too, and the map $a(Z_y)$ is a map between contractible spaces. The hypothesis insures that if $Z_y(\emptyset)$ is empty then for some singleton $T$ the space $Z_y(T)$ is empty, too, so that $a(Z_y)$ is a map between empty spaces.
\end{proof}

The next two results, from section 2 of \cite{Goodwillie_CALC2}, concern the interplay between cartesian and cocartesian. We sometimes refer to them as higher Blakers-Massey theorems.

\begin{thm}\label{Blakers-Massey-n-ad}
Let $X\: \mathcal P(S)\rightarrow \cal T$ be an $r$-cube of spaces. Assume that for each nonempty subset $T\subset S$ there is given an integer (or $+\infty$) $k_T$ such that the front face $\partial^TX$ is $k_T$-cocartesian. Assume also the monotonicity condition: $k_U\leq k_T$ when $U\subset T$. Then the cube $X$ is $k$-cartesian where $k+r-1$ is the minimum, over all partitions $S=\coprod_{\alpha}T_{\alpha}$ of $S$, of the sum $\Sigma_{\alpha}k_{\alpha}$.
\end{thm}

An important special case is that in which $k_T=+\infty$ for all subsets $T$ having more than one element, in other words the case in which every $2$-dimensional face of $X$ is $\infty$-cocartesian, in other words the case in which the cube is equivalent to one that is made by pushout from $n$ cofibrations with a common domain. (In \cite{Goodwillie_CALC2} and \cite{Goodwillie_CALC3} these cubes are called strongly cocartesian.) In this case the statement is simply that the cube is $(1-r+\Sigma_{i=1\dots r}k_i)$-cartesian where $k_i$ is the connectivity of the map $X(\emptyset)\rightarrow X(\lbrace i\rbrace)$.

This special case implies the general case by an argument given in
\cite{Goodwillie_CALC2}. In the proof of the Theorem in \cite{Goodwillie_CALC2} the special case for $r$ is deduced from the general case for $r-1$.

There is also a dual result, whose proof is somewhat easier:

\begin{thm} \label{dual}
Let $X\:\mathcal P(S)\rightarrow \cal T$ be an $r$-cube of spaces. Assume that for each nonempty subset $T\subset S$ there is given an integer (or $+\infty$) $k_T$ such that the back face $\partial_{S-T}X$ is $k_T$-cartesian. Assume also the monotonicity condition: $k_U\leq k_T$ when $U\subset T$. Then the cube $X$ is $k$-cocartesian where $k-r+1$ is the minimum, over all partitions $S=\coprod_{\alpha}T_{\alpha}$ of $S$, of the sum $\Sigma_{\alpha}k_{\alpha}$.
\end{thm}

Again in the special case of a strongly cartesian cube (all $2$-dimensional faces $\infty$-cartesian) the statement is simpler: The cube is $(r-1+\Sigma_{i=1}^{n}k_i)$-cocartesian where $k_i$ is the connectivity of the map $X(S-\lbrace i\rbrace)\rightarrow X(S)$.

As indicated in the introduction, \ref{Blakers-Massey-n-ad} can be used to prove a very weak form of the main conjecture about smooth embedding spaces:

\begin{prop}\label{weak}
Let $N$ be a smooth compact $n$-dimensional manifold. Let
$(Q_1,\dots,Q_r)$, $r\geq 1$, be pairwise disjoint compact
submanifolds of $N$ transverse to $\partial N$ and let $q_i$ be the dimension of $Q_i$. Let $(P,\partial_0 P,\partial_1P)$ be a compact manifold triad and
let $p$ be the dimension of $P$. Suppose that an embedding $e_0:\partial_0 P\rightarrow \partial N$
is given, disjoint from $\partial_0Q_i$ for all $i$. For $S\subset
\underline r = \lbrace 1,\dots r\rbrace$ let $Q_S=\cup_{i\in S}Q_i$.
In this situation the $r$-dimensional cubical diagram formed by the spaces $\lbrace E(P,N-Q_S)\rbrace$ is $(1+\Sigma_i(n-p-q_i-2))$-cartesian.
\end{prop}

\begin{proof}
Here is the argument in the case $r=2$. Consider the embedding space $E(P,N)$ and its open subspaces $E(P,N-Q_1)$ and $E(P,N-Q_2)$. Let $k_1=n-p-q_1-1$, $k_2=n-p-q_2-1$, and $k_{12}=2n-2p-q_1-q_2-1$. The pair
$$
(E(P,N),E(P,N-Q_1)\cup E(P,N-Q_2))
$$
is $k_{12}$-connected because for a generic $s$-parameter family of embeddings $P\rightarrow N$ every embedding in the family will miss either $Q_1$ or $Q_2$ or both as long as $s<2n-2p-q_1-q_2$. This makes the square $k_{12}$-cocartesian, because the union of two open subsets is not only their pushout along the intersection but also their homotopy pushout (since the pushout square is a homology pushout by the excision theorem for singular homology, and it is a homotopy pushout because the
fundamental group condition is taken care of by the van Kampen
theorem). The same kind of dimension-counting shows that the pairs
$$
(E(P,N-Q_2),E(P,N-Q_1)\cap E(P,N-Q_2))
$$
and
$$
(E(P,N-Q_1),E(P,N-Q_1)\cap E(P,N-Q_2))
$$
are respectively $k_1$-connected and $k_2$-connected. The smaller of $k_{12}$ and $k_1+k_2$ is $k_1+k_2$, so the square is $(k_1+k_2-1)$-cartesian.

Now consider the general case. Dimension-counting shows that the cube associated with the $r$-ad
$$
(E(P,N);E(P,N-Q_1),\dots,E(P,N-Q_r))
$$
is $((n-p-q_1)+\dots+(n-p-q_r)-1)$-cocartesian, and more generally that each front face is $k_S$-cocartesian for the appropriate $S$, where $k_S=\Sigma_{i\in S}(n-p-q_i)-1$. Again the partition of $\underline r$ by singletons is the worst case, and
\ref{Blakers-Massey-n-ad} gives that the cube is $((n-p-q_1)+\dots+(n-p-q_r)-2r+1)$-cartesian.

Because of the monotonicity hypothesis in \ref{Blakers-Massey-n-ad},
the argument just given only applies in cases where $n-p-q_i\geq 0$ for all $i$. But these cases imply the rest in a trivial way, using induction over $r$: If some term, say $n-p-q_r$, is negative, then view the $r$-cube as a map of the $(r-1)$-cubes corresponding to $(N;Q_1,\dots,Q_{r-1})$ and $(N-Q_r;Q_1,\dots,Q_{r-1})$. By induction those two cubes are $k$-cartesian where $
k=(n-p-q_1)+\dots +(n-p-q_{r-1})-2(r-1)+1$. This makes the $r$-cube $(k-1)$-cartesian by \ref{simple_rule_1}, and therefore
$((n-p-q_1)+\dots +(n-p-q_r)-2r+1)$-cartesian.
\end{proof}

\section{\label{poincare} Appendix C: Poincar\'e Spaces}

We review the definitions of Poincar\'e duality space, pair, and triad. For present purposes very broad definitions will suffice: there is no need to require any finiteness or to mention Whitehead torsion. Such matters will necessarily make a brief appearance in \cite{Good_Klein}
when we make the transition to manifolds.

\begin{defn}The space $X$ satisfies $n$-dimensional duality (or is a Poincar\'e space of formal dimension $n$) if there exist a coefficient system $\cal L$ of infinite cyclic groups on $X$ (the \it dualizing system\rm) and a homology class $\lbrack X\rbrack\in H_n(X;\cal L)$ (the \it fundamental class\rm) such that for every coefficient system $\cal G$ on $X$ and every $p\in \Bbb Z$ the map
$$
H^p(X;\cal G)\rightarrow H_{n-p}(X;\cal L\otimes\cal G).
$$
given by cap product with  $\lbrack X\rbrack$ is an isomorphism. A Poincar\'e complex is a CW Poincar\'e space.
\end{defn}

A space weakly equivalent to a Poincar\'e space is again a Poincar\'e space. In order for a space to satisfy $n$-dimensional duality it is necessary and sufficient that it has a finite number of path components and that each of these satisfies $n$-dimensional duality.

The dimension is determined by $X$ (unless $X$ is empty) because it is the largest $n$ such that for some local system on $X$ the $n^{th}$ homology group is nontrivial. The dualizing system $\cal L$ is unique up to isomorphism because among all local systems of infinite cyclic groups on $X$ it is the only one for which the $n^{th}$ homology group is isomorphic to $H^0(X;\Bbb Z)$. The ordered pair $(\cal L, \lbrack X\rbrack)$ (the \it duality data\rm ) is unique up to \it unique\rm\ isomorphism because it is the universal example of a local system on $X$ and an $n$-dimensional homology class. This is so because for any local system $\cal G$ the group $\hom(\cal L,\cal G)$ of maps from $\cal L$ to $\cal G$ is naturally isomorphic to
$$
H^0(X;\hom_{\Bbb Z}(\cal L,\cal G))\cong H_n(X;\cal L\otimes_{\Bbb Z}\hom_{\Bbb Z}(\cal L,\cal G))\cong H_n(X;\cal G).
$$
Some authors include dualizing data as part of the structure of a Poincar\'e space and require compatibility of dualizing data as part of the definition of Poincar\'e embedding. By using the strong uniqueness principle above (and its generalization to pairs), we are keeping some irrelevant details out of our definitions and proofs.

\begin{defn}A map $\partial X\rightarrow X$ of spaces satisfies $n$-dimensional duality (or is an $n$-dimensional Poincar\'e map) if there exist a coefficient system $\cal L$ of infinite cyclic groups on $X$ and a homology class $\lbrack X,\partial X\rbrack\in H_n(X,\partial X;\cal L)$ such that
\begin{enumerate}
\item for every coefficient system $\cal G$ on $X$ and every $p$ the cap product map
$$
H^p(X;\cal G)\rightarrow H_{n-p}(X,\partial X;\cal L\otimes\cal G)
$$
is an isomorphism,
\item  for every coefficient system $\cal G$ on $X$ and every $p$ the cap product map
$$
H^p(X,\partial X;\cal G)\rightarrow H_{n-p}(X;\cal L\otimes\cal G)
$$
is an isomorphism, and
\item for every coefficient system $\cal G$ on $\partial X$ and every $p$ the map
$$
H^p(\partial X;\cal G)\rightarrow H_{n-p-1}(\partial X;(\cal L|_{\partial X})\otimes\cal G)
$$
given by cap product with $\partial_*\lbrack X,\partial X\rbrack\in H_{n-1}(\partial X;\cal L|_{\partial X})$ is an isomorphism.
\end{enumerate}A Poincar\'e pair is a CW pair $(X,\partial X)$ such that the inclusion $\partial X\rightarrow X$ satisfies $(n-1)$-dimensional duality.
\end{defn}

There is some redundancy in the definition above. By the five lemma, (1) follows from (2) and (3), and likewise (2) follows from (1) and (3). In general (3) does not follow from (1) and (2), but it does so if every coefficient system on $\partial X$ is the restriction of some system on $X$, for example if the map $\partial X\rightarrow X$ is $2$-connected. In any case, in the presence of (1) and (2), (3) is implied by the weaker statement that $\partial X$ satisfies $(n-1)$-dimensional duality. Indeed, given (1) and (2), any dualizing system for $\partial X$ must be isomorphic to $\cal L|_{\partial X}$ because the five lemma gives $H_{n-1}(\partial X;\cal L|_{\partial X})\cong H^0(X;\Bbb Z)$; and it is not hard to see that then $\partial_*\lbrack X,\partial X\rbrack$ must be a fundamental class

As in the absolute case, the ordered pair $(\cal L,\lbrack X,\partial X\rbrack)$ is universal ($\hom(\cal L,\cal G)\cong H_n(X,\partial X;\cal G)$) and therefore unique up to unique isomorphism.

Duality is preserved by gluing along a common boundary:

\begin{prop} \label{gluing_pd}
Given $n$-dimensional Poincar\'e pairs $(P,\partial P)$ and $(C,\partial C)$ with $\partial C=\partial P$, the pushout $N=P\cup_{\partial P}C$ is an $n$-dimensional Poincar\'e complex.
\end{prop}

\proof Consider the pushout square
$$
\xymatrix{
 \partial P\ar[r]\ar[d] & C\ar[d]\\
P\ar[r] & N
}
$$
If $(\cal L^P,\lbrack P,\partial P\rbrack)$ and $(\cal L^C,\lbrack C,\partial P\rbrack)$ are duality data for the pairs then there is a unique isomorphism between $\cal L^P|_{\partial P}$ and $\cal L^C|_{\partial P}$ taking $\partial_*\lbrack P,\partial P\rbrack$ to $-\partial_*\lbrack C,\partial P\rbrack$. This leads to duality data $(\cal L^N,\lbrack N\rbrack)$ for $N$ as follows. The system $\cal L^N$ on $N$ is chosen, using that isomorphism, so that it restricts to $\cal L^P$ and $\cal L^C$. The class $\lbrack N\rbrack\in H_n(N;\cal L^N)$ is chosen to map to the image of $\lbrack P,\partial P\rbrack$ in $H_n(N,C;\cal L^N)$ and to the image of $\lbrack C,\partial P\rbrack$ in $H_n(N,P;\cal L^N)$, using an exact sequence
$$
0\rightarrow H_n(N;\cal L^N)\rightarrow H_n(N,P;\cal L^N)\oplus H_n(N,C;\cal L^N)\rightarrow  H_{n-1}(\partial P;\cal L^N)
$$
The following diagram makes possible a five-lemma argument to show that the cap product with $\lbrack N\rbrack$ is an isomorphism. We have
$$
\xymatrix{
{}\ar[d] & {}\ar[d]\\
H^p(N,C)\ar[r]\ar[d] & H_{n-p}(P)\ar[d]\\
H^p(N)\ar[r]\ar[d] & H_{n-p}(N)\ar[d]\\
H^p(C)\ar[r]\ar[d] & H_{n-p}(N,P)\ar[d]\\
H^{p+1}(N,C)\ar[r]\ar[d] & H_{n-p-1}(P)\ar[d]\\
{} & {}
}
$$
where all cohomology is with coefficients in a given system $\cal G$ on $N$ (or its restriction) and all homology is with coefficients in $\cal L^N\otimes \cal G$. The first horizontal arrow is the composition of $\lbrack P,\partial P\rbrack\cap$ with an excision isomorphism, the second is $\lbrack N\rbrack\cap$, and the third is the composition of an excision isomorphism with $\lbrack C,\partial C\rbrack\cap$. The first square commutes because the two composed maps $H^p(N,C)\rightarrow H_{n-p}(N)$ are given by cap product with the same element of $H_n(N,C;\cal L^N)$; by construction $\lbrack N\rbrack$ has the same image in $H_n(N,C;\cal L^N)$ as $\lbrack P,\partial P\rbrack$. The second square commutes because the two composed maps $H^p(N)\rightarrow H_{n-p}(N,P)$ are given by cap product with the same element of $H_n(N,P;\cal L^N)$; again by construction $\lbrack N\rbrack$ has the same image in $H_n(N,P;\cal L^N)$ as $\lbrack C,\partial P\rbrack$. The third square commutes!
 up to sign because the two composed maps may both be expressed as
$$
H^p(C)\rightarrow H^p(\partial P)\rightarrow H_{n-p-1}(\partial P)\rightarrow H_{n-p-1}(P)
$$
with the center arrow representing cap product with $\partial_*\lbrack P,\partial P\rbrack$ in one case and with $\partial_*\lbrack C,\partial P\rbrack$ in the other.
\endproof

Of course, it follows that more generally if
$$
\xymatrix{
 \partial P\ar[r]\ar[d] & C\ar[d]\\
P\ar[r] & N
}
$$
is a homotopy pushout square and both ${\partial P\rightarrow P}$ and $\partial P\rightarrow C$ are $n$-dimensional Poincar\'e maps then $N$ is an $n$-dimensional Poincar\'e space.

If $\partial P\rightarrow P$ and $N$ are given, both satisfying $n$-dimensional duality, then a square as above is called a Poincar\'e embedding of $P$ in $N$. In other words, a Poincar\'e embedding consists of a map $P\rightarrow N$ and a factorization $\partial P\rightarrow C\rightarrow N$ of the composition of ${\partial P\rightarrow P\rightarrow N}$ that satisfies the following two conditions: the resulting square is a homotopy pushout, and the map $\partial P\rightarrow C$ satisfies $n$-dimensional duality.

\ref{gluing_pd}  says that if $P$ and $C$ satisfy duality then $P\cup C$ satisfies duality. We now address the  converse question: if $P$ and $P\cup C$ are known to satisfy duality, how can one tell whether $C$ satisfies duality? In other words, how does one recognize a Poincar\'e embedding?

Notice that if the square above is a homology pushout then duality for $N$ and for $\partial P\rightarrow P$ leads to a map $\cal L^P\rightarrow \cal L^N|_P$ of coefficient systems on $P$, namely the element of
$$
\hom(\cal L^P,\cal L^N|_P)\cong H^0(P;\hom_{\Bbb Z}(\cal L^P,\cal L^N|_P))\cong H_n(P,\partial P;\cal L^N|_P)
$$
that corresponds to the image of $\lbrack N\rbrack$ in $H_n(N,C;\cal L^N)$.

\begin{prop}  \label{homological_test}
Suppose that the square
$$
\xymatrix{
 \partial P\ar[r]\ar[d] & C\ar[d]\\
P\ar[r] & N
}
$$
is a homotopy pushout, and that both the space $N$ and the map ${\partial P\rightarrow P}$ satisfy $n$-dimensional duality, and that the map $\partial P\rightarrow P$ is $2$-connected. Then the following conditions are equivalent:
\begin{enumerate}
\item  The map $\partial P\rightarrow C$ satisfies $n$-dimensional duality; that is, the square constitutes a Poincar\'e embedding.
\item The map $\cal L^P\rightarrow \cal L^N|_P$ described above is an isomorphism of coefficient systems on $P$.
\item  There is an isomorphism $\cal L^N|_P\cong \cal L^P$ such that the image of $\lbrack N\rbrack$ in $H_n(N,C;\cal L^N)$ coincides with the image of $\lbrack P,\partial P\rbrack$.
\end{enumerate}
\end{prop}

\proof Certainly $(1) \Rightarrow (2) \Rightarrow (3)$, even without the assumption of $2$-connectedness. For $(3) \Rightarrow (1)$ we proceed as follows. To produce duality data for $(C,\partial P)$, let $\cal L^C$ be the restriction of $\cal L^N$ and let $\lbrack C,\partial P\rbrack\in H_n(C,\partial P;\cal L^C)$ be the unique element that maps to the image of $\lbrack N\rbrack$ in $H_n(N,P;\cal L^N)$. Use the same ladder diagram as in the proof of
\ref{gluing_pd}. The first square commutes this time because the needed compatibility between $\lbrack N\rbrack$ and $\lbrack P,\partial P\rbrack$ has been assumed. The second commutes because $\lbrack C,\partial P\rbrack$ has been chosen so as to have the needed compatibility with $\lbrack N\rbrack$. The third commutes up to sign because again $\partial_*\lbrack P,\partial P\rbrack=-\partial_*\lbrack C,\partial P\rbrack$. Duality for $N$ and for $P$ gives two isomorphisms, so by the five-lemma the map $H^p(C)\rightarrow H_{n-p}(C,\partial P)$ given by cap product with $\lbrack C,\partial P\rbrack\in H_n(C,\partial P;\cal L^C)$ is an isomorphism. This is valid for all coefficient systems on $C$ that extend to $N$. In view of the assumption of $2$-connectedness, that means all systems.
\endproof

To summarize, $C$ satisfies duality if both $P$ and $P\cup C$ satisfy duality, as long as the fundamental classes of $P$ and $N$ are compatible and the $2$-connectedness holds.

The condition on fundamental classes cannot be dispensed with in general. For example, in the pushout square
$$
\xymatrix{
S^{n-1}\ar[r]\ar[d] & S^{n-1}\vee N\ar[d]\\
D^n\ar[r] & D^n\vee N
}
$$
if $N\simeq D^n\vee N$ satisfies Poincar\'e duality there is no reason why the pair $(S^{n-1}\vee N,S^{n-1})$ should do so.

One may also speak of a Poincar\'e embedding of $P$ in $N$ when $P$ is a $p$-dimensional duality space with $p<n$. This means a codimension zero Poincar\'e embedding of $(P,\partial P)$ in $N$ where $\partial P\rightarrow P$ is a map whose homotopy fibers are
weakly equivalent to the sphere $S^{n-p-1}$. (Such a map is
necessarily an $n$-dimensional duality pair.) We are not taking that
approach here.

We briefly recall how some of these notions extend
from pairs to triads.

\begin{defn}A square diagram
$$
\xymatrix
{
\partial_{12}X\ar[r]\ar[d] & \partial_1X\ar[d]\\
\partial_2X\ar[r]
& X
}
$$
satisfies $n$-dimensional Poincar\'e duality if both
$\partial_{12}X\rightarrow \partial_1X$ and
$\partial_{12}X\rightarrow \partial_2X$ are $(n-1)$-dimensional
Poincar\'e maps and $$
\hocolim (\partial_1
X\leftarrow\partial_{12}X\rightarrow \partial_2X)\rightarrow X
$$
is
an $n$-dimensional Poincar\'e map. A Poincar\'e triad is a CW triad
such that the associated square
is a Poincar\'e square.
\end{defn}

A
five-lemma argument yields cap product isomorphisms
$$
H^p(X,X_1;\cal G)\cong H_{n-p}(X,X_2;\cal L\otimes \cal
G).
$$

Proposition \ref{gluing_pd} generalizes to apply to gluing two Poincar\'e
triads along part of the boundary. We omit the proof, which involves
no new ideas. The statement is:

\begin{prop}If two $n$-dimensional
Poincar\'e triads $(P;\partial_0 P,\partial_1P)$ and $(C;\partial_1
C,\partial_2C)$ are such that $(\partial_1
P,\partial_{01}P)=(\partial_1 C,\partial_{12}C)$, then
$$(P\cup_{\partial_1P}C,\partial_0P\cup_{\partial_{01}P}\partial_2C)$$
is an $n$-dimensional Poincar\'e pair.
\end{prop}

We record the
equivalent statement involving homotopy pushouts and no cofibrancy
hypothesis: In a
cube
$$
\xymatrix{
(\partial_{01}P\rightarrow\partial_1P)\ar[r]\ar[d]
& (\partial_2C\rightarrow C)\ar[d]\\
(\partial_{0}P\rightarrow
P)\ar[r] & (\partial N\rightarrow N)
}
$$
if both of the
squares
$$
\xymatrix{
\partial_1P\ar[r]\ar[d]& C\ar[d]\\
P\ar[r] &
N
}
$$
$$
\xymatrix{
\partial_{01}P\ar[r]\ar[d]&
\partial_2C\ar[d]\\
\partial_0P\ar[r] & \partial N
}
$$
are homotopy
pushouts and
$$
\xymatrix{
\partial_{01}P\ar[r]\ar[d]&
\partial_1P\ar[d]\\
\partial_0P\ar[r] &
P
}
$$
$$
\xymatrix{
\partial_{01}P\ar[r]\ar[d]&
\partial_1P\ar[d]\\
\partial_2C\ar[r] & C
}
$$
 are Poincar\'e
squares then $\partial N\rightarrow N$ is a Poincar\'e map.

Such a
cubical diagram may be called a Poincar\'e embedding of $P$ in $N$.
Our point of view here is that the Poincar\'e triad $P$ and the
Poincar\'e pair $N$ are given, as
well as the embedding
$$
\xymatrix{
\partial_{01}P\ar[r]\ar[d]&
\partial_2C\ar[d]\\
\partial_0P\ar[r]& \partial N
}
$$
of $\partial_0P$ in $\partial N$.  To give an embedding of $P$ in $N$ is then to

specify a map $P\rightarrow N$ compatible with the given map

$\partial _0P$ in $\partial N$
and a space $C$ with a suitable

factorization
$$
\partial_1P\cup_{\partial_{01}P}\partial_2C\rightarrow C\rightarrow N
$$

The homological test

\ref{homological_test} extends to this setting:

\begin{prop}

Suppose that in the cube
$$
\xymatrix{
(\partial_{01}P\rightarrow\partial_1P)\ar[r]\ar[d]& (\partial_2C\rightarrow C)\ar[d]\\
(\partial_{0}P\rightarrow P)\ar[r] & (\partial N\rightarrow N)
}
$$
both of the squares
$$
\xymatrix{
\partial_1P\ar[r]\ar[d]& C\ar[d]\\
P\ar[r] & N
}
$$
$$
\xymatrix{
\partial_{01}P\ar[r]\ar[d]& \partial_2C\ar[d]\\
\partial_0P\ar[r] & \partial N
}
$$
are homotopy pushouts and both the
square
$$
\xymatrix{
\partial_{01}P\ar[r]\ar[d]& \partial_1P\ar[d]\\
\partial_0P\ar[r] & P
}
$$
and the map $\partial N\rightarrow N$
satisfy $n$-dimensional duality. Then the following conditions are equivalent:
\begin{enumerate}
\item The square
$$
\xymatrix{
\partial_{01}P\ar[r]\ar[d]& \partial_1P\ar[d]\\
\partial_2C\ar[r] & C
}
$$
satisfies $n$-dimensional duality; that is, the cube constitutes a Poincar\'e
embedding.
\item The canonical map $\cal L^P\rightarrow \cal
L^N|_P$
is an isomorphism of coefficient systems on $P$.
\item
There is an isomorphism $\cal L^N|_P\cong \cal L^P$ such that the
image of
$\lbrack N\rbrack$ in $H_n(N,C;\cal L^N)$ coincides with the image of
$\lbrack P,\partial P\rbrack$.
\end{enumerate}
\end{prop}

The proof is the same as before.

It is worth noting, although we do not use it
in this paper, that in many cases the homological test is
superfluous:

\begin{prop}In the situation above, if the map
$\partial_0P\rightarrow P$ is $0$-connected then the diagram does
constitute a Poincar\'e embedding.
\end{prop}

\proof We verify
condition (2). The map $\cal L^N\rightarrow \cal L^P$ in question
gives an isomorphism at each point of $\partial_0P$, because the
square
$$
\xymatrix{
\partial_{01}P\ar[r]\ar[d]&
\partial_2C\ar[d]\\
\partial_0P\ar[r]& \partial N
}
$$
is a Poincar\'e embedding. (Use $(1)\implies (2)$ for the square.) It
follows that it does so at each point of $P$.
\endproof

Finally, here is one more case where the homology test can be bypassed.
Briefly, $C$ satisfies duality if $P$, $P\cup C$, $Q$, and $C\cup Q$
satisfy duality. We state it first in the case when there is no
extraneous boundary.

\begin{prop} \label{poincare-gluing} Let
$(P,\partial P)$ and $(Q,\partial Q)$ be $n$-dimensional Poincar\'e
pairs and let $(C,\partial C)$ be a CW pair such that $\partial C$
is the disjoint union of $\partial P$ and $\partial Q$. Assume that
the pairs $(P\cup C,\partial Q)$ and $(C\cup Q,\partial P)$
are both $n$-dimensional Poincar\'e. Then the pair
$(C,\partial C)$ is also
$n$-dimensional Poincar\'e, provided the inclusion map
$\partial P\rightarrow P$ is $2$-connected.
\end{prop}

Let us
ponder this
statement before we prove it. The space $N=P\cup C\cup
Q$ has two
reasons for being $n$-dimensional Poincar\'e, because it
is both the
union of $P\cup C$ and $Q$ along their common boundary
and the union
of $P$ and $C\cup Q$ along their common boundary.
According to one
reading of the conclusion, in order for the pushout
square
$$
\xymatrix {
\partial P\cup \partial Q\ar[r]\ar[d] &
C\ar[d]\\
P\cup Q \ar[r] & N
}
$$
to be a Poincar\'e embedding of
$P\cup Q$ in $N$ it is sufficient if it gives Poincar\'e embeddings
of both $P$ and $Q$ in $N$ (as long as the homotopy spine dimension
of $P$ is at most $n-3$).
According to another, it says that in
order for the left-hand pushout square in
$$
\xymatrix{
\partial P\ar[r]\ar[d] & C\ar[d]\ar[r] & C\cup Q\ar[d]\\
P \ar[r] & P\cup C\ar[r] & P\cup C\cup Q
}
$$
to be a Poincar\'e embedding of $P$ in
$N-Q=P\cup C$ it is sufficient if the large square is a Poincar\'e
embedding of $P$ in $N=P\cup C\cup Q$ (again as long as the homotopy
spine dimension of $P$, or for that matter $Q$, is at most
$n-3$).

\proof Duality in $P$ and in $P\cup C$ gives a map of
coefficient systems on $P$:
$$
\mathcal L^P\rightarrow\mathcal L^{P\cup C}.
$$
Duality in $P\cup C$ and in $P\cup C\cup Q$ gives a
map
$$
L^{P\cup C}\rightarrow\mathcal \mathcal L^{P\cup C\cup Q},
$$
which is an isomorphism because of duality in $Q$. Duality in
$P$ and in $P\cup C\cup Q$ gives a map
$$
L^P\rightarrow\mathcal L^{P\cup C\cup Q},
$$
which is an isomorphism because of
duality in $C\cup Q$, and which is also the composition of the other
two maps. It follows that the first map is an isomorphism. By
\ref{homological_test}, duality holds for $C$.
\endproof

The same
result, with the same proof, is valid more generally if $P$ and $Q$
are allowed to have some shared boundary, and if each of $P$, $Q$,
and $C$ is allowed to have boundary that is not shared with either
of the others. In the most general case, $N=P\cup C\cup Q$ satisfies
duality relative to $\partial N$, $P$ satisfies duality relative to
$P\cap(\partial N\cup  C\cup Q)$, $Q$ satisfies duality relative to
$Q\cap(\partial N\cup P\cup C)$, and it follows that $C$ satisfies
duality relative to $C\cap(\partial N\cup P\cup Q)$.


\begin{thebibliography}{GKW}
\bibliographystyle{invent}

\bibitem[DK]{DK}%
Dwyer, W.G., Kan D.M.: A classification theorem for diagrams
of simplicial
sets.
\newblock {\it Topology\rm } {\bf 23} (1984),
139--155.

\bibitem[G]{thesis}%
Goodwillie, T.G.:
A multiple
disjunction lemma for smooth concordance embeddings.
\newblock {\it
Memoirs AMS\rm}
{\bf 86}
(1990).

\bibitem[G2]{Goodwillie_CALC2}%
Goodwillie, T.G.:
Calculus.
II. Analytic functors.
\newblock {\it $K$-theory\rm }
{\bf 5} (1991/92),
295--332.

\bibitem[G3]{Goodwillie_CALC3}%
Goodwillie,
T.G.: Calculus
III: Taylor series
\newblock {\it Geometry and
Topology\rm } {\bf 7}
(2003),  645--711.

\bibitem[GK]{Good_Klein}%
Goodwillie, T.G.,
Klein J.R.: Multiple disjunction for spaces
of smooth embeddings.
\newblock In
preparation.

\bibitem[GKW]{Good_Klein_Weiss}%
Goodwillie, T.G.,
Klein J.R., Weiss, M.S.: Spaces of smooth embeddings, disjunction
and
surgery.
\newblock {\it Surveys on surgery theory,\rm} Vol. 2,
221--284,
\newblock Ann. of Math. Stud., 149, Princeton Univ.
Press,
Princeton, NJ,
2001.


\bibitem[GW1]{Good_Weiss1}%
Goodwillie, T.G.,
Weiss,
M.:
Embeddings from the point of view of immersion theory. I.
\newblock {\it Geom. Topol.}  {\bf 3}  (1999),
67--101.

\bibitem[GW2]{Good_Weiss2}%
Goodwillie, T.G., Weiss,
M.:
Embeddings from the point of view of immersion theory. II.
\newblock {\it Geom. Topol.}  {\bf 3}  (1999),
103--118.


\bibitem[K]{Klein}%
Klein, J.~R.: {Poincar\'e embeddings
and fiberwise homotopy theory}.
\newblock {\it Topology} {\bf 38}
(1999), 597--620.

\bibitem[Q]{Quillen}%
Quillen, D.: {Homotopical
Algebra}.
\newblock (LNM, Vol.~43).
\newblock Springer
1967


\bibitem[Wa]{W}%
Waldhausen, F.: Algebraic $K$-theory of
spaces.
\newblock {\it Algebraic and Geometric Topology, Proceedings
Rutgers 1983,
LNM 1126}, 1985, pp.\
318--419.

\bibitem[We]{Weiss}%
Weiss, M.:
Calculus of embeddings.
\newblock {\it Bull. Amer. Math. Soc.\rm}  {\bf 33}  (1996),
177--187.

\end{thebibliography}
\end{document}